\documentclass[12 pt]{amsart}
\usepackage{amsaddr}
\usepackage{mathtools,amsmath,amssymb,amsthm,enumerate, xcolor, hyperref}
\newtheorem{theorem}{Theorem}[section]
\newtheorem{corollary}[theorem]{Corollary}
\newtheorem{lemma}[theorem]{Lemma}
\newtheorem{proposition}[theorem]{Proposition}
\theoremstyle{definition}
\newtheorem{definition}[theorem]{Definition}

\theoremstyle{remark}
\newtheorem{remark}[theorem]{Remark}

\newcommand{\cQ}{\mathcal{Q}}
\newcommand{\cK}{\mathcal{K}}
\DeclareMathOperator{\cZ}{\mathcal{Z}}

\DeclareMathOperator{\id}{id}
\DeclareMathOperator{\Ad}{Ad}
\DeclareMathOperator{\Aut}{Aut}
\DeclareMathOperator{\Inn}{Inn}

\DeclareMathOperator{\Out}{Out}
\DeclareMathOperator{\fin}{fin}
\DeclareMathOperator{\SL}{SL}
\DeclareMathOperator{\actson}{\curvearrowright}
\DeclareMathOperator{\cS}{\mathcal{S}}
\DeclareMathOperator{\cU}{\mathcal{U}}
\DeclareMathOperator{\cR}{\mathcal{R}}
\DeclareMathOperator{\cC}{\mathcal{C}}
\DeclareMathOperator{\cD}{\mathcal{D}}
\DeclareMathOperator{\PSL}{PSL}
\DeclareMathOperator{\GL}{GL}

\DeclareMathOperator{\spec}{Spec}
\DeclareMathOperator{\dotdot}{\cdot\cdot\cdot}

\title{A TYPE $\textrm{III}_{1}$ FACTOR WITH THE SMALLEST OUTER AUTOMORPHISM GROUP} 
\author{Soham Chakraborty}
\address{Department of Mathematics, KU Leuven \\
02.32, Celestijnenlaan 200B, Leuven 3001, Belgium}
\email{soham.chakraborty@kuleuven.be}
\usepackage[doi=false, url =false , isbn = false,style=alphabetic,sorting=nyt, backend = biber]{biblatex}
\begin{document}
\setlength{\parindent}{0em}
\setlength{\parskip  }{5.5 pt}

\begin{abstract}
    The canonical modular homomorphism provides an embedding of $\mathbb{R}$ into the outer automorphism group $\Out(M)$ of any type $\rm{III}_{1}$ factor $M$. We provide an explicit construction of a full factor of type $\rm{III}_{1}$ with separable predual such that the outer automorphism group is minimal, i.e. this embedding is an isomorphism and a homeomorphism. We obtain such a $\rm{III}_{1}$ factor by using an amalgamated free product construction.  
\end{abstract}

\maketitle

\section{Introduction}

In this article we will only deal with von Neumann algebras with separable preduals. An important invariant for the study of $\rm{II}_{1}$ factors is the outer automorphism group of the factor. Recall that an automorphism of such a factor $M$ is called inner if it is of the form $\Ad(u)$ for a unitary $u \in M$. The outer automorphism group $\Out(M)$ is defined as the quotient of the group of all automorphisms of $M$ by the group of inner automorphisms. In \cite{MR0103421}, Blattner showed that $\Out(R)$ for the hyperfinite $\rm{II}_{1}$ factor $R$ is huge: it contains an isomorphic copy of every locally compact second countable group. In general, computation of the outer automorphism group for a $\rm{II}_{1}$ factor is a difficult problem. The first significant step in this direction was taken in \cite{MR0587372} where Connes introduced the notion of property (T) von Neumann algebras and showed that any property (T) factor has countable outer automorphism group. In particular he showed that group von Neumann algebras $L(G)$ of property (T) groups $G$ with infinite conjugacy classes (ICC) have countable outer automorphism groups. An interesting and natural question then was if there exists a $\rm{II}_{1}$ factor $M$ with trivial outer automorphism group. This was a difficult question and remained unanswered for many years until Popa's deformation and rigidity theory made it possible to attack such problems. 

In their seminal article \cite{MR2386109}, Ioana, Peterson and Popa showed that any second countable compact abelian group appears as the outer automorphism group of a $\rm{II}_{1}$ factor. In \cite{MR2409162}, Falgui\`eres and Vaes improved this result and showed that in fact every compact group (not necessarily abelian) arises as outer automorphism group of a $\rm{II}_{1}$ factor. However all these results used a Baire category argument which showed the existence of such $\rm{II}_{1}$ factors, but did not provide explicit constructions. In \cite{MR2601038}, Popa and Vaes gave explicit constructions of certain $\rm{II}_{1}$ factors with prescribed outer automorphism groups, including the first explicit example of a $\rm{II}_{1}$ factor with trivial outer automorphism group. It is still unknown if every locally compact unimodular group arises as the outer automorphim group of a $\rm{II}_{1}$ factor. However by now we have a large class of examples of such groups: in \cite{MR2504433}, Vaes showed that every discrete countable group arises as an outer automorphism group, we earlier saw that every compact group arises this way too. In \cite{deprez2012explicit}, Deprez showed that the groups $\SL^{\pm}(n,\mathbb{R})$ also arise as outer automorphism groups of $\rm{II}_{1}$ factors.

This article deals with the type $\rm{III}$ version of this problem. If $M$ is a type $\rm{III}$ factor with a faithful normal state $\phi$, then the modular automorphism group $(\sigma^{\phi}_{t})_{t \in \mathbb{R}}$ form a natural class of automorphisms of $M$. A deep result of Connes showed that the modular automorphism group does not depend on the choice of the faithful normal state up to inner automorphisms. If furthermore $M$ is of type $\rm{III}_{1}$, then the modular automorphisms are outer, and hence we have that $\mathbb{R} \subseteq \Out(M)$. We are interested in the question: can we explicitly construct a factor $M$ of type $\rm{III}_{1}$ such that $\Out(M) = \mathbb{R}$. We tackle this problem with ergodic theoretic methods: by using group measure space von Neumann algebras, where the outer automorphism group has a more tractable description in terms of the outer automorphism group of the associated equivalence relation, as below. 

Suppose that we have a free ergodic nonsingular action $\Gamma \curvearrowright (Z_{1},\eta_{1})$ for a discrete countable group $\Gamma$ and let the associated countable ergodic equivalence relation on $(Z_{1},\eta_{1})$ be $\mathcal{R}_{1}$ (the subscript 1 is used intentionally as we deal with actions of type $\rm{III}_{1}$). Recall that an automorphism of $\mathcal{R}_{1}$ is a nonsingular Borel isomorphism $\theta: Z_{1} \rightarrow Z_{1}$ such that for a.e. $z,z' \in Z_{1}$, $(z,z') \in \mathcal{R}_{1}$ if and only if $(\theta(z), \theta (z')) \in \mathcal{R}_{1}$. Such an automorphism is called inner if $(z, \theta(z)) \in \mathcal{R}_{1}$ for a.e. $z \in Z_{1}$. The outer automorphism group $\Out(\mathcal{R}_{1})$ is defined similarly as the quotient of the automorphism group by the inner ones. By the theory of Feldman and Moore (\cite{MR578656}, \cite{MR578730}), denoting by $M_{1}$ the von Neumann algebra $L^{\infty}(Z_{1},\eta_{1}) \rtimes \Gamma$ and by $A$ the Cartan subalgebra $L^{\infty}(Z_{1},\eta_{1})$, we have that: \begin{equation}
\label{intro eq 1}
    \Out(A \subset M_{1}) = \Out(\mathcal{R}_{1}) \ltimes H^{1}(\mathcal{R}_{1},\mathbb{T})
\end{equation}
where $\Out(A \subset M_{1})$ denotes the group of outer automorphisms preserving the Cartan subalgebra $A$ and $H^{1}(\mathcal{R}_{1},\mathbb{T})$ denotes the quotient of the 1-cocycle group of $\mathcal{R}_{1}$ with values in the circle, by the group of coboundaries. In particular if $A \subset M_{1}$ is the unique Cartan subalgebra up to unitary conjugacy, then equation \ref{intro eq 1} indeed becomes: 
\begin{equation}
    \Out(M_{1}) = \Out(\mathcal{R}_{1}) \ltimes H^{1}(\mathcal{R}_{1},\mathbb{T})
\end{equation}
The modular automorphism group sits inside the 1-cohomolgy group $H^{1}(\mathcal{R}_{1},\mathbb{T})$. Therefore the idea is to construct a `rigid' free ergodic action $\Gamma \curvearrowright (Z_{1},\eta_{1})$ of type $\rm{III}_{1}$ such that every orbit equivalence is inner, the 1-cohomology group with values in the circle is equal to $\mathbb{R}$ and the associated von Neumann algebra has a unique Cartan subalgebra up to unitary conjugacy. In \cite{MR3335839}, Ioana showed that for a large class of amalgamated free product groups, any free ergodic pmp action has the property that the crossed product has a unique Cartan subalgebra. Notice that if $\Gamma = \Gamma_{1} \ast_{\Sigma} \Gamma_{2}$ is an amalgamted free product group and $\Gamma \curvearrowright (Z, \eta)$ is a nonsingular free ergodic action which is induced from an action $\Gamma_{i} \curvearrowright (X, \mu)$, then $L^{\infty}(Z) \rtimes \Gamma$ is isomporphic to $B(\mathcal{H}) \overline{\otimes} (L^{\infty}(X) \rtimes \Gamma_{i})$, and in such a situation we cannot have uniqueness of Cartan subalgebras. However this can be taken care of by imposing some recurrence of $\Gamma_{i} \curvearrowright (Z, \eta)$ relative to $\Sigma$. Using the methods of \cite{MR3102166}, Vaes in \cite[Theorem 8.1]{Vaes14} proved that under some suitable recurrence assumptions, any non singular free ergodic action has the property that the crossed product has a unique Cartan subalgebra, generalizing the unique Cartan decomposition results in \cite{MR3164361}. We use a slightly modified version of this result that we state in Theorem \ref{Thm: on uniqueness of Cartan from Vaes}, which was already observed in \cite[Remark 8.3]{Vaes14}. 

In this article we construct such an action of an amalgamated free product group that is sufficiently rigid. Orbit equivalence (OE)-superrigidity results for an action usually follow from cocycle superrigidity with countable target groups. Over the years a lot of these results have been proven for example for Bernoulli actions of property (T) groups in \cite{MR2342637}, of product groups in \cite{MR2425177} and for profinite actions of property (T) groups in \cite{MR2783933}. In \cite{PopaVaes11} the authors use similar methods to prove cocycle and OE-superrigidity results for actions of lattices in $\SL(n,\mathbb{R})$ on $\mathbb{R}^{n}$. Since these actions are not probability measure preserving any more, the notion of property (T) has to be replaced by Zimmer's notion of property (T) for non singular actions, and this plays a very important role in \cite{PopaVaes11}. Using such an action of a lattice and the unique Cartan results, Vaes in \cite[Proposition D]{Vaes14} constructed the first example of a type III action which is W*-superrigid. In \cite[Remark 6.10]{MR4664830}, the authors used this construction to give examples of type $\rm{III}_{0}$ actions that are W*-superrigid with any prescribed flow. Our construction in this article is similar to the one in \cite[Proposition D]{Vaes14} but the main difficulty is that we need to make suitable modifications to take care of $\Out(\mathcal{R}_{1})$ and the 1-cohomology group $H^{1}(\mathcal{R}_{1},\mathbb{T})$.  

Now we explain our construction. We take two countable discrete groups  $\Gamma_{1}$ and $\Lambda$ and let $\Sigma < \Gamma_{1}$ be a proper subgroup. We consider the amalgamated free product $\Gamma = \Gamma_{1} *_{\Sigma} (\Sigma \times \Lambda)$ with the natural quotient morphism $\pi: \Gamma \rightarrow \Gamma_{1}$. We will consider two actions: a free ergodic probability measure preserving action of $\Gamma \curvearrowright (X,\mu)$ and a free ergodic nonsingular action $\Gamma_{1} \curvearrowright (Y_{1},\nu)$ of type $\textrm{III}_{1}$ such that the infinite measure preserving Maharam extension $\Gamma_{1} \curvearrowright (Y_{1} \times \mathbb{R}, \nu \times \lambda)$ is ergodic (here $\lambda$ is the Lebesgue measure on $\mathbb{R}$). Putting $(Z_{1}, \eta_{1}) = (X \times Y_{1}, \mu \times \nu)$, we show that the diagonal action $\Gamma\curvearrowright Z_{1}$ given by $g \cdot (x,y) = (gx, \pi(g) \cdot y)$ satisfies the required properties. We take $\Gamma_{1} \curvearrowright Y_{1}$ to be the rigid lattice action $\SL(6,\mathbb{Z})\curvearrowright \mathbb{R}^{6}/\mathbb{R}^{*}_{+}$ such that the Maharam extension $\SL(6,\mathbb{Z}) \curvearrowright \mathbb{R}^{6}$ is ergodic. We then pick a subgroup $H < \Gamma$ in a way such that the conjugates of $\Sigma$ does not intersect $H$ except at $\pm{I}$ and take $\Gamma \curvearrowright X$ to be the generalized Bernoulli action $\Gamma \curvearrowright \{0,1\}^{\Gamma/ H}$ with an unequal measure on the atomic base space $\{0,1\}$. In Section 4 we show that the Maharam extension of this action is cocycle superrigid with any target group in Popa's class of $\mathcal{U}_{\fin}$ groups (see Definition \ref{Def Cocycle Superrigid}). We essentially show that any 1 cocycle restricted to $\SL(6,\mathbb{Z})$ is cohomologous to a group homomorphism using \cite[Theorem 5.3]{PopaVaes11}. The reason we choose $H$ to not intersect conjugates of $\Sigma$ is so that the action $\Sigma \curvearrowright X$ remains mixing, which we need to show cocycle superrigidity. Finally we use the fact that the group $\Gamma$ is perfect to deduce that the 1 cohomology group is isomorphic to $\mathbb{R}$.   

In Section 5 we finally compute the outer automorphism group. We apply cocycle superrigidity to the Zimmer cocycle of an orbit equivalence $\Delta$ of the Maharam extension. We use a general cocycle superrigidity to OE-superrigidity result proven in \cite[Lemma 2.4]{DrimbeVaes23} to see that $\Delta(g \cdot z) = \delta(g) \cdot \Delta(z)$ for an automorphism $\delta \in \Aut(\Gamma)$. We then use some properties of automorphisms of amalgamated free product groups as in Lemma \ref{Lemma: automorphisms of amalgamated free products} and Lemma \ref{Lemma: delta2 is identity} to deduce that $\delta$ is trivial and $\Delta$ is indeed a conjugacy. In Proposition \ref{Prop: every conjugacy is trivial} we note that $\Delta$ is trivial on the generalized Bernoulli action and on $\mathbb{R}^{6}$, we have that $\Delta(y) = g \cdot y$ for some some $g$ normalizing $\SL(6,\mathbb{Z})$ in $\GL(6,\mathbb{R})$ as in \cite[Theorem 6.2]{PopaVaes11}. Thus we deduce that $\Delta$ is of the form $(\id \times \rho_{\lambda})$ where $\rho_{\lambda}(y) = \lambda y$ for a non zero real number $\lambda$ in $\mathbb{R}^{6}$. Finally we note in Theorem \ref{main theorem} that the induced automorphism of $\mathbb{R}^{6}/\mathbb{R}^{*}_{+}$ is inner. Together with the fact that the 1 cohomology group is equal to $\mathbb{R}$ and that the Cartan subalgebra $L^{\infty}(Z_{1}, \eta_{1})$ is unique up to unitary conjugacy, we have the required $\rm{III}_{1}$ factor with the smallest outer automorphism group. 

\section*{Acknowledgements}
The author is supported by FWO research project G090420N of the Research Foundation Flanders. The author would like to thank his PhD supervisor Stefaan Vaes for his immense support throughout the process of writing this article. The author would also like to thank Bram Verjans for a lot of helpful discussions on the topic.

\section{Preliminaries}
\subsection{Facts about amalgamated free products}
We begin this section with the following definition:

\begin{definition}
    \label{Const: Sigma, Lambda and Gamma}
    Let $\Gamma_{1} \coloneqq \SL(6, \mathbb{Z})$ and let $\Sigma$ be the subgroup of $\SL(6, \mathbb{Z})$ given by 

\begin{equation}
\Sigma \coloneqq 
\left(\begin{array}{@{}c|c@{}}
  \pm 1
  & \mathbb{Z}^{1 \times 5} \\
\hline
  0 &
  \GL(5, \mathbb{Z})
\end{array}\right) \cap  \SL(6, \mathbb{Z}) \end{equation}
We put $\Lambda \coloneqq \SL(3,\mathbb{Z})$ and consider the amalgamated free product group $\Gamma \coloneqq \Gamma_{1} *_{\Sigma} (\Sigma \times \Lambda)$ with the natural quotient map $\pi: \Gamma \rightarrow \Gamma_{1}$.
\end{definition}

We shall now state several results about these groups that we need to prove our main theorem. We begin with the following well-known result about automorphisms of linear groups. A proof of this appears in the literature in several places, for example in \cite{MR0049194}.

\begin{proposition}
\label{Prop: outer automorphism groups of linear groups}
    Let $G = \SL(n,\mathbb{Z})$ for $n \geq 3$. Then the only outer automorphisms of $G$ are as follows: 
    \begin{enumerate}
        \item If $n$ is odd, $|\Out(G)| = 2$, and is generated by the order 2 element $\phi \in \Aut(G)$ defined by $\phi(A) = (A^{T})^{-1}$ for all $A \in G$. 
        \item If $n$ is even, $|\Out(G)| = 4$ and is generated by the element $\phi \in \Aut(G)$ as above and the element $\psi \in \Aut(G)$ defined as $\psi(A) = BAB^{-1}$ for all $A \in G$ and a fixed element $B \in \GL(n, \mathbb{Z})$ with determinant -1.
    \end{enumerate}
\end{proposition}

\begin{definition}
    \label{Def: sigma_tilde}
Let $\Sigma < \Gamma_{1}$ as in Definition \ref{Const: Sigma, Lambda and Gamma} and let $e_{1}$ denote the column vector with 1 in the first row and 0 elsewhere. We define the subgroup $\widetilde{\Sigma}$ of $\GL(6,\mathbb{Z})$ as $\widetilde{\Sigma} \coloneqq \{B \in \GL(6, \mathbb{Z}) \; | \; B(e_{1}) = \pm e_{1}\}$. Notice that $\widetilde{\Sigma}$ can be described as follows: 
\begin{equation}
    \widetilde{\Sigma} = 
\left(\begin{array}{@{}c|c@{}}
  \pm 1
  & \mathbb{Z}^{1 \times 5} \\
\hline
  0 &
  \GL(5, \mathbb{Z})
\end{array}\right)
\end{equation}
\end{definition}

\begin{proposition}
\label{Prop: Group homomorphisms preserving Sigma are Inner} 
Any automorphism $\phi$ of $\SL(6,\mathbb{Z})$ such that $\phi(\Sigma) = \Sigma$ is of the form $\phi = \Ad(B)$ with $B \in \widetilde{\Sigma}$.
\end{proposition}
\begin{proof}
   First we note the following: The globally invariant proper subspaces of $\mathbb{R}^{6}$ under the subgroup $\Sigma^{T}$ has to be of dimensions 5. On the other hand the globally invariant proper subspaces of $\mathbb{R}^{6}$ under $\Sigma$ must have dimension 1. Since conjugation does not change the dimension of a globally invariant subspace, we have that $\Sigma^{T}$ cannot be conjugated onto $\Sigma$. By Proposition \ref{Prop: outer automorphism groups of linear groups}, we have that such an automorphism has to be of the form $A \mapsto BAB^{-1}$ with $B \in \GL(6,\mathbb{Z})$ and $\det(B) = \pm 1$. If $W \subset \mathbb{R}^{6}$ is a one dimensional globally invariant subspace for $\Sigma$, then $B\cdot W$ is a globally invariant subspace for $B\Sigma B^{-1}$. Thus if $B\Sigma B^{-1} = \Sigma$ for some $B \in \GL(6,\mathbb{Z})$ then $B \cdot W = W$ and this forces $B$ to be in $\widetilde{\Sigma}$.     
\end{proof}

The rest of this section will be dedicated to constructing a suitable subgroup $H$ of $\Gamma$ as in Definition \ref{Const: II infinity action}. 
\begin{definition}
    \label{Def: elements A_i}
    For a natural number $k \geq 2$, let $p_{0} = 5$, $p_{1} = 7$ and $p_{2},...,p_{k} \geq 11$ be distinct prime numbers. We define now the elements $A_{0}$, $A_{1}$ and $A_{i}$ for $2 \leq i \leq k$ of $\Lambda = \SL(3,\mathbb{Z})$:   
\begin{equation}
        A_{0} = \begin{pmatrix}
    1 & 0 & 0 \\
    0 & p_{0}-2 & 1 \\ 
    0 & -1 & 0
    \end{pmatrix}, \text{ } 
    A_{1} = \begin{pmatrix}
    p_{1} - 2 & 0 & 1 \\
    0 & 1 & 0 \\ 
    -1 & 0 & 0
    \end{pmatrix}, \text{ } 
    A_{i} = \begin{pmatrix}
    p_{i}- 2 & 1 & 0 \\
    -1 & 0 & 0 \\ 
    0 & 0 & 1
    \end{pmatrix}    
\end{equation}
\end{definition}

\begin{lemma}
\label{Lemma: no non trivial homomorphisms of lambda}
    Let $\mathcal{W} = \{A_{j}^{n_{j}} \; | \; 0 \leq j \leq k, 0 \neq n_{j} \in \mathbb{Z}\}$. If $\sigma \in \Aut(\SL(3, \mathbb{Z}))$ and $\sigma(A_{i}) \in \mathcal{W}$ for all $i \in \{0,...,k\}$, then $\sigma = \id$. 
\end{lemma}
\begin{proof}
    Solving the characteristic polynomials, we note that the matrices $A_{i}$ have three distinct eigenvalues: 1, $\frac{p_{i} - 2 \pm \sqrt{p_{i}(p_{i} - 4)}}{2}$. Since the $p_{i}$'s are distinct, two of these eigenvalues are nontrivial algebraic integers in distinct degree 2 field extensions of $\mathbb{Q}$. Denoting by $\spec(A)$ the set of eigenvalues of a matrix $A$, we note the following: 
    \begin{itemize}
        \item $\spec(A_{i}) = \spec(A_{i}^{-1})$ for all $i$
        \item If $i \neq j$, for any nonzero integer $k_{i}$ and any integer $k_{j}$, we have that $\spec(A_{i}^{k_{i}}) \neq \spec(A_{j}^{k_{j}})$
        \item For a fixed $i$, if $k \neq l$, we have that $\spec(A_{i}^{k}) \neq \spec(A_{i}^{l})$.
    \end{itemize}
    By Proposition \ref{Prop: outer automorphism groups of linear groups}, we have that $\sigma$ is either of the form $\Ad(Y)$ for $Y \in \SL(3, \mathbb{Z})$ or of the form $\Ad(Y) \circ \phi$ where $\phi(A) = (A^{T})^{-1}$. In particular we see that $\spec(A) = \spec(\sigma(A))$ for all $A \in \SL(3,\mathbb{Z})$. By the previous paragraph, we have hence that $\sigma(A_{i}) = A_{i}^{\pm 1}$ for all $i$. Now suppose first that $\sigma = \Ad(Y) \circ \phi$ and consider the matrix 
    \begin{equation}
     B = \begin{pmatrix}
        0 & 1 & 0 \\ -1 & 0 & 0 \\ 0 & 0 & 1
    \end{pmatrix}   
    \end{equation}
    and notice that for all $i \geq 2$, we have $\phi(A_{i}) = BA_{i}B^{-1}$. Then letting $C = YB$, we have that $\sigma(A_{i}) = CA_{i}C^{-1}$. Thus for all $i \geq 1$, $C$ either commutes with $A_{i}$ or $CA_{i}C^{-1} = A_{i}^{-1}$. Thus $C$ preserves globally the 1-dimensional eigenspace corresponding to the eigenvalue 1 and also preserves the 2-dimensional eigenspace corresponding to the two other eigenvalues of $A_{i}$. Thus $C$ has a simple block decomposition into a $2 \times 2$ block and a $1 \times 1$ block and by multiplying with $B^{-1}$, we know that $Y$ has the following form: 
    \begin{equation}
    Y = \begin{pmatrix}
        * & * & 0 \\
        * & * & 0 \\
        0 & 0 & \pm 1
    \end{pmatrix}    
    \end{equation}
    Now since $C$ also either commutes with $A_{0}$ or $CA_{0}C^{-1} = A_{0}^{-1}$, by the same argument we have that $Y$ is also necessarily of the following form: 
    \begin{equation}
    Y = \begin{pmatrix}
        \pm 1 & 0 & 0 \\
        0 & * & * \\
        0 & * & *
    \end{pmatrix}    
    \end{equation}
    From the two equations of $Y$, we can conclude that $Y$ is a diagonal matrix with all entries equal to $\pm 1$ in the diagonal. But then we have that $Y(A_{0}^{T})^{-1}Y^{-1} \neq A_{0}^{\pm 1}$, hence giving us a contradiction. Thus $\sigma$ cannot be of the form $\Ad(Y) \circ \phi$. Suppose now that $\sigma = \Ad(Y)$. By a similar argument as earlier, we get once again that $Y$ is a diagonal matrix with $\pm 1$'s in the diagonal. Let us denote the diagonal entries of $Y$ by $\epsilon_{i}$ for $i = 1,2,3$. For $i \geq 2$ we can then check that $YA_{i}Y^{-1}$ has to be equal to $A_{i}$ and this forces $\epsilon_{1} = \epsilon_{2}$. Similarly we can calculate $YA_{1}Y^{-1}$ and check that it is necessarily equal to $A_{1}$ and this forces $\epsilon_{1} = \epsilon_{3}$. Since $Y$ has determinant 1, this implies that $Y = I$.      
\end{proof}

\begin{lemma}
\label{Lemma: Sigma is non-normal}
    For $n \geq 3$, consider the subgroup $\Sigma$ of $\SL(n, \mathbb{Z})$ given by $\Sigma = \{g \in \SL(n,\mathbb{Z}) \; | \; g(e_{1}) = \pm e_{1}\}$. As before consider the subgroup $\widetilde{\Sigma}$ of $\GL(n,\mathbb{Z})$ as $\widetilde{\Sigma} = \{g \in \GL(n, \mathbb{Z}) \; | \; g(e_{1}) = \pm e_{1}\}$. There exist elements $g_{1},...,g_{r} \in \SL(n, \mathbb{Z}) \backslash \Sigma$ such that 
    \begin{equation}
        \bigcap_{i=1}^{r} g_{i}^{-1}\widetilde{\Sigma} g_{i} = \{\pm I\}
    \end{equation}
    Moreover the $g_{i}$'s can be chosen such that $g_{i}g_{j}^{-1} \notin \Sigma$ if $i \neq j$. 
\end{lemma}

\begin{proof}
    For $2 \leq i \leq n$, we define  the elements $\sigma_{i}$ and $\alpha_{i}$ of $\SL(n,\mathbb{Z}) \backslash \Sigma$ as the following linear operators:
\begin{equation}
    \sigma_{i}:
    \begin{cases}
        e_{1}  \mapsto -e_{i}\\
        e_{i} \mapsto e_{1} \\
        e_{j} \mapsto e_{j} \text{ for } j \neq \{1,i\}
    \end{cases}
\end{equation}

\begin{equation}
    \alpha_{i}:
    \begin{cases}
        e_{1}  \mapsto e_{1} + 2e_{i}\\
        e_{i} \mapsto e_{1} + 3e_{i} \\
        e_{j} \mapsto e_{j} \text{ for } j \neq \{1,i\}
    \end{cases}
\end{equation}
Now restricted to the linear subspace $\mathbb{C}e_{1} + \mathbb{C}e_{i}$, the matrices for the operators $\sigma_{i}$ and $\alpha_{i}$ are given by $\left(\begin{array}{cc}
    0 & 1 \\
    -1 & 0
\end{array} \right)$ and $\left(\begin{array}{cc}
    1 & 1 \\
    2 & 3
\end{array} \right)$ respectively. We claim now that 
\begin{equation}
    (\bigcap_{i=2}^{n}\sigma_{i}^{-1}\widetilde{\Sigma}\sigma_{i}) \cap (\bigcap_{j=2}^{n}\alpha_{j}^{-1}\widetilde{\Sigma}\alpha_{j})  =  \{\pm I\}   
\end{equation}
    Clearly we have that $\sigma_{i}^{-1} \widetilde{\Sigma}\sigma_{i} = \{g \in \GL(n,\mathbb{Z}) \; | \; g(e_{i}) = \pm e_{i} \}$. Consider now an element $A \in (\bigcap_{i=2}^{n}\sigma_{i}^{-1}\widetilde{\Sigma}\sigma_{i}) \cap (\bigcap_{j=2}^{n}\alpha_{j}^{-1}\widetilde{\Sigma}\alpha_{j})$. In particular, $A \in \sigma_{i}^{-1} \widetilde{\Sigma} \sigma_{i}$ which then forces $A(e_{i}) = \pm e_{i}$ for all $2 \leq i \leq n$ and we denote by $\epsilon_{i}$ the element in $\pm \{1\}$ such that $A(e_{i}) = \epsilon_{i}e_{i}$. We now write $A(e_{1}) = \epsilon_{1}e_{1} + \sum_{i=2}^{n} x_{i}e_{i}$ with the coefficients $x_{i} \in \mathbb{Z}$ and note that $\epsilon_{1} = \pm 1$. For every $j > 1$, we have that $\alpha_{j}A\alpha_{j}^{-1} \in \widetilde{\Sigma}$. Thus $\alpha_{j}A\alpha_{j}^{-1}(e_{1}) = \pm e_{1}$. It is easy to see that: 
\begin{equation}
    \alpha_{j}^{-1}:
    \begin{cases}
        e_{1}  \mapsto 3e_{1} - 2e_{j}\\
        e_{j} \mapsto -e_{1} + e_{j} \\
        e_{k} \mapsto e_{k} \text{ for } k \neq \{1,j\}
    \end{cases}
\end{equation}
Then one can explicitly compute the values of $\alpha_{j}A\alpha_{j}^{-1}(e_{1})$ and see that
for $i \neq \{1,j\}$, $x_{i} = 0$. Since this holds for all $j$, we have that $x_{i} = 0$ for all $i \geq 2$. Also, one can check that the coefficient of $e_{j}$ in $\alpha_{j}A\alpha_{j}^{-1}(e_{1})$ is equal to $6(\epsilon_{1} - \epsilon_{j})$ and hence $\epsilon_{1} - \epsilon_{j} = 0$ for all $j$ and consequently $A = \pm I$. For $i \neq j$ and $k, l$, it follows that $\sigma_{i}\sigma_{j}^{-1}$, $\alpha_{i}\alpha_{j}^{-1}$ or $\sigma_{k}\alpha_{l}^{-1}$ are not elements of $\Sigma$ and that concludes the proof of our claim. 
\end{proof}

\begin{definition}
\label{Def: subgroup H}
    For $k \geq 2$, we choose elements $g_{1},...,g_{k} \in \SL(6,\mathbb{Z}) \backslash \Sigma$ as in Lemma \ref{Lemma: Sigma is non-normal} such that $g_{i}g_{j}^{-1} \notin \Sigma$ for $i \neq j$. We define $H$ as the subgroup of $\Gamma$ generated by $\pm I \in \Sigma$, the elements $g_{i}^{-1}A_{i}g_{i}$ for $i = 1,...,k$ and $A_{0}$. Recall that $g_{i} \in \Gamma_{1}$ for all $1 \leq i \leq k$ and $A_{i} \in \Lambda$ for all $0 \leq i \leq k$ as in Definition \ref{Def: elements A_i}. Notice that $H$ is isomorphic to $\mathbb{Z}/2\mathbb{Z} \times \mathbb{F}_{k+1}$. 
\end{definition}   

Now we state some preliminaries about reduced words in an amalgamated free products $G = G_{1} *_{K} (K \times L)$, where $G_{1}, L$ and $K$ are non-trivial discrete countable groups with $K \neq G_{1}$. We look at some applications of such reduced words in computing their automorphism groups. 

\begin{definition}
\label{Def: reduced words}
    Any word of the form $a_{0}b_{1}a_{1}b_{2}a_{2}\dotdot b_{n}a_{n}$ with $n \geq 1$ where $a_{i} \in L$ with $a_{j} \neq e$ for $j =1,...,n-1$ and $b_{i} \in G_{1} \backslash K$ for all $i$ is called \textit{reduced expression} in $G$.
\end{definition}

\begin{remark}
\label{Remark: reduced words uniqueness}
We note that every element of $G \backslash (K \times L)$ can be written in this form. However such expressions are not unique. In fact two such expressions $a_{0}b_{1}a_{1}\dotdot b_{n}a_{n}$ and $c_{0}d_{1}c_{1}\dotdot d_{m}c_{m}$ define the same element in $G $ if and only if $n =m$, $a_{i} = c_{i}$ for all $i$ and there exists $\sigma_{1},...,\sigma_{n-1} \in K$ such that: 
\begin{enumerate}
    \item $b_{1} = d_{1}\sigma_{1}$
    \item $\sigma_{i-1}b_{i} = d_{i}\sigma_{i}$ for $i \in \{2,...,n-1\}$
    \item $\sigma_{n-1}b_{n} = d_{n}$
\end{enumerate}
This automatically forces $c_{0}d_{1}c_{1}\dotdot d_{i}\sigma_{i} = a_{0}b_{1}a_{1}\dotdot b_{i}$ in $G$ for all $i \in \{1,...,n-1\}$. For an element $g \in G$, we define the length $|g|$ as the length of a reduced expression of $g$. By convention, we say that an element in $K$ has length 0.
\end{remark}

Now recall the Definitions of the groups $H$, $\Gamma_{1}$, $\Gamma$, $\Sigma$ and $\Lambda$. Comparing reduced expressions of the elements in $g^{-1}Hg$ for $g \in \Gamma$, we get the following result: 

\begin{lemma}
    \label{Lemma: conjugates of sigma dont intersect H}
For any element $g \in \Gamma$, we have that $g\Sigma g^{-1} \cap H \subseteq \{\pm I\}$.  
\end{lemma}
\begin{proof}
    Recall that $H$ is generated by $\{\pm I\}$, $A_{0}$ and the elements $g_{i}^{-1}A_{i}g_{i}$. Since we can choose $g_{i}$'s such that $g_{i}g_{j}^{-1} \notin \Sigma$, a reduced expression of an element in $H$ is either a positive power of $A_{0}$ or has length at least 3. In either cases we have that $H$ does not intersect $\Sigma$ except at $\{\pm I\}$. Now suppose we have $g = \sigma \lambda \in \Sigma \times \Lambda$ and $h \in H$ such that $h \notin \{\pm I\}$. Then $ghg^{-1} \in \Sigma$ if and only if $\lambda h \lambda^{-1} \in \Sigma$. Once again $\lambda h \lambda^{-1}$ has a reduced expression of length at least 5 or $h$ is a power of $A_{0}$ and $\lambda h \lambda^{-1} \in \Lambda$: either ways, $\lambda h \lambda^{-1} \notin \Sigma$. Now let $g \notin \Sigma \times \Lambda$ and let $g = a_{0}b_{1}a_{1}...b_{n}a_{n}$ be a reduced expression for $g$ with $b_{i} \in \Gamma \backslash \Sigma$ and $a_{i} \in \Lambda$. First suppose that $h = A_{0}^{n}$, then $\lambda = a_{n}ha_{n}^{-1} \in \Lambda$ and $|ghg^{-1}| \geq |b_{n}\lambda b_{n}^{-1}| \geq 3$. So without loss of generality (up to inverses of letters), we can assume that a reduced expression of $h$ is of the form $\pm g_{i_{1}}A_{i_{1}}g_{i_{1}}^{-1}...g_{i_{k}}A_{i_{k}}g_{i_{k}}^{-1}$ for indices $i_{1},i_{2},...,i_{k}$ where the letters $g_{i}$'s and $A_{i}$'s vary over $1 \leq i \leq k$.
    If $a_{n} \neq e$ then the reduced expression of $ghg^{-1}$ clearly has length equal to $2|g| + |h|$ and hence it cannot be in $\Sigma$. So we are left with the case when $a_{n} = e$ and hence $g = a_{0}b_{1}a_{1}...b_{n}$.
    
    First let us assume $n > 1$. Notice that if $b_{n}g_{i_{1}} \notin \Sigma$ and $b_{n}g_{i_{k}} \notin \Sigma$, once again we have that the length of $ghg^{-1}$ is equal to $2|g| + |h| - 2$ and once again it cannot be in $\Sigma$. If $\sigma = b_{n}g_{i_{1}} \in \Sigma$, then $b_{n-1}a_{n}b_{n}hg_{i_{1}} = b_{n-1}a_{n}\sigma = b_{n-1}\sigma a_{n}$ and $b_{n-1}\sigma$ cannot be in $\Sigma$ as that would imply $b_{n-1} \in \Sigma$. As a result (applying a symmetric argument for $b_{n}g_{i_{k}}$), we have that the length of $ghg^{-1}$ is greater than $|h|$ and hence it is not in $\Sigma$. If $n = 1$, since $\Lambda$ commutes with $\Sigma$, we can assume $g \in \Gamma \backslash \Sigma$. If $ghg^{-1} \in \Sigma$ and $gg_{i_{1}} \in \Sigma$, then $|ghg^{-1}| = |h| - 2$ which is still at least 1. If $gg_{i_{1}} \notin \Sigma$ then $|ghg^{-1}| = |h| + 2$, and in both cases we can easily conclude that $ghg^{-1} \notin \Sigma$ hence concluding our proof. 
\end{proof}

We now state the following rather ad-hoc result about automorphisms of general amalgamated free product groups that we need for proving the main result of this paper. Note that for a general amalgamated free product group $G_{1} *_{K} G_{2}$, one can similarly define reduced words exactly as in Definition \ref{Def: reduced words} by asking that $a_{i} \in G_{2} \backslash K$ and $a_{j} \neq e$. One can then define the \textit{length} of an element $g \in G$, denoted by $|g|$ as the length of a reduced word (which is again well defined by the same argument as in Remark \ref{Remark: reduced words uniqueness}). 

\begin{lemma}
\label{Lemma: automorphisms of amalgamated free products}
    Suppose $G_{i},G'_{i}$ for $i \in \{1,2\}$ are discrete groups with property $(T)$. Further suppose $K \neq G_{i}$ and $K'\neq G'_{i}$ for $i=1,2$ and $\delta': G = G_{1} *_{K} G_{2} \rightarrow G'_{1} *_{K'} G'_{2} = G'$ is an isomorphism between the amalgamated free product groups. Then $\delta'$ is the composition of an inner automorphism and $\delta$ where $\delta$ satisfies exactly one of the following: 
    \begin{enumerate}
        \item $\delta(G_{1}) = G'_{1}$, $\delta(G_{2}) = G'_{2}$ and $\delta(K) = K'$.
        \item $\delta(G_{1}) = G'_{2}$, $\delta(G_{2}) = G'_{1}$ and $\delta(K) = K'$.
    \end{enumerate}
\end{lemma} 

\begin{proof}
     As in the proof of Part 4 in \cite[Theorem 3.1]{DeprezVaes11} for a subgroup $G_{0} < G$, letting $|G_{0}| \coloneqq \sup \{|g|, \; g \in G\}$, we have that $|G_{0}| < \infty$ if and only if there exists $g \in G_{0}$ such that $gG_{0}g^{-1} \subseteq G_{1}$ or $gG_{0}g^{-1} \subseteq G_{2}$. Also as in \cite[Theorem 3.1] {DeprezVaes11}, if $G_{0}$ has property (T), then $|G_{0}| < \infty$. Applying this to the group $\delta'(G_{2})$, we can assume without loss of generality that we have an element $g \in G'$ such that $\Ad(g) \circ \delta'(G_{2}) \subseteq G'_{2}$. Then let $\delta_{1} = \Ad(g) \circ \delta'$ and so that $\delta_{1}(G_{2}) \subseteq G'_{2}$. Now since $G_{1}$ is also a property (T) group, there exists $h \in G'$ such that $\Ad(h) \circ \delta_{1}(G_{1}) \subseteq G'_{1}$ or $\Ad(h) \circ \delta_{1}(G_{1}) \subseteq G'_{2}$. However we have that $K' \neq G'_{1}$ and $K' \neq G'_{2}$ and hence the whole group $G'$ can not be generated by $G'_{2} \bigcup h^{-1} G'_{2}h$. Thus it is necessary that $\Ad(h) \circ \delta_{1}(G_{1}) \subseteq G'_{1}$.   
     
    Now notice that $\delta_{1}(G_{1})$ and $\delta_{1}(G_{2})$ generates $G'$ and hence in particular every element of $G_{1}'$ has to be generated by elements of $h^{-1}G'_{1}h$ and $G'_{2}$. One can check that this forces such an element $h$ to be of the form $lk$ where $k \in G'_{1}$ and $l \in G'_{2}$. Now letting $\delta = \Ad(k) \circ \delta_{1}$, we have that $\delta(G_{1}) \subseteq G'_{1}$ and $\delta(G_{2}) \subseteq G'_{2}$. Since $\delta$ is surjective, we moreover have that $\delta(G_{1}) = G'_{1}$ and $\delta(G_{2}) = G'_{2}$. It immediately follows that $\delta(K) = K'$. On the other hand if we started with a $g \in G'$ such that $\Ad(g) \circ \delta(G_{2}) \subseteq G'_{1}$, then similarly we would have an isomorphism $\delta$ such that $\delta(G_{1}) = G'_{2}$, $\delta(G_{2}) = G'_{2}$ and $\delta(K) = K'$.      
\end{proof}

\begin{corollary}
    \label{Corr: automorphisms of Gamma}
    Let $\Gamma$ as in Definition \ref{Const: Sigma, Lambda and Gamma} and let $\delta \in \Aut(\Gamma)$. Then after composing with an inner automorphism, $\delta$ is of the form $\delta_{1} \ast (\delta_{1}|_{\Sigma} \times \delta_{2})$ where $\delta_{1} \in \Aut(\Gamma_{1})$ and $\delta_{2} \in \Aut(\Lambda)$. 
\end{corollary}
\begin{proof}
    By Lemma \ref{Lemma: automorphisms of amalgamated free products}, we have that after multiplying by an inner automorphism, either $\delta(\Gamma_{1}) = \Gamma_{1}$ or $\delta(\Gamma_{1} \times \Lambda)$. If $\delta(\Gamma_{1}) = \Sigma \times \Lambda$, then denoting the coordinate projections in $\Sigma \times \Lambda$ by $\pi_{1}$ and $\pi_{2}$, we have that $\pi_{2} \circ \delta$ is a homomorphism from $\SL(6,\mathbb{Z})$ to $\SL(3,\mathbb{Z})$. By \cite[Lemma 3]{Weinberger97}, such a homomorphism needs to be trivial and hence $\delta(\Gamma_{1}) \subseteq \Sigma$ which is absurd. Hence $\delta(\Gamma_{1}) = \Gamma_{1}$ and $\delta(\Sigma \times \Lambda) = \Sigma \times \Lambda$ and $\delta$ preserves the amalgam $\Sigma$. 

    \textit{Claim 1: The center $\cZ(\Sigma)$ of $\Sigma$ is equal to $\{\pm I\}$.} To prove Claim 1, suppose that an element $\sigma \in \Sigma$ is of the form: 
    \begin{align*}
       \sigma = 
        \left(\begin{array}{@{}c|c@{}}
        \epsilon
        & v \\
        \hline
        0 &
        A
        \end{array}\right)
    \end{align*}
    where $A \in \GL(5,\mathbb{Z})$, $v = (v_{1},v_{2},...,v_{5})$ is a row vector in $\mathbb{Z}^{5}$ and $\epsilon = \pm 1$. Notice that if $\epsilon = 1$, then $\det(A)$ has to be $1$ and hence $A \in \SL(5,\mathbb{Z})$. Since $A$ has to commute with all elements in $\SL(5,\mathbb{Z})$, we have that $A \in \cZ(SL(5,\mathbb{Z}))$. Since $\SL(5,\mathbb{Z})$ is simple (see \cite[Example 1.2.4]{DeMedtsLectureNotes}), we have that $A = I$. Similarly if $\epsilon = -1$, notice that $\det(A)$ has to be $-1$, which implies that $-A \in \SL(5,\mathbb{Z})$ and similarly, $A = -I$. Now for $1 \leq i \leq 6$ let $\rho_{i}$ be matrices in $\Sigma$ that have 1's along the diagonal, 1 in the position (row i, column 6) and 0 everywhere else. Since $\sigma$ commutes with $\rho_{i}$'s, one can check that this implies $v = 0$, thus proving the claim.

    Now we proceed with the proof of the lemma, notice that the only thing to actually prove is that $\delta(\Lambda) = \Lambda$. Let $\pi_{1}$ and $\pi_{2}$ be the coordinate projections in $\Sigma \times \Lambda$. Consider the group homomorphisms $\phi \coloneqq \pi_{1} \circ \delta : \Sigma \times \Lambda \rightarrow \Sigma$ and let $\Sigma_{0}$ be $\phi(\Lambda)$. We will show that $\Sigma_{0} \subseteq \cZ(\Sigma)$. If $\sigma_{0} = \phi( \lambda_{0}) \in \Sigma_{0}$, take any element $\phi(\sigma) \in \Sigma$. Notice that every element of $\Sigma$ is of this form as $\delta: \Sigma \rightarrow \Sigma$ is an isomorphism. Then we have: 
    \begin{align*}
       \phi(\sigma) \sigma_{0} \phi(\sigma)^{-1} = \phi(\sigma \lambda_{0} \sigma^{-1}) = \phi(\lambda_{0}) = \sigma_{0}
    \end{align*}
    thus proving our claim. Now from Claim 1, $\Sigma_{0} = \{I\}$ or $\Sigma_{0} = \{\pm I\}$. Now looking at the homormorphism $\phi|_{\Lambda} :\Lambda \rightarrow \Sigma$, we have that the image cannot be $\{ \pm I\}$ because $\Lambda$ is a simple group and the kernel would be a nontrivial proper subgroup. Hence $\phi|_{\Lambda} = \id$ and hence $\delta(\Lambda) = \Lambda$, as needed.  
\end{proof}

\begin{lemma}
\label{Lemma: delta2 is identity}
    Let $\Gamma, H $ be as in Definition \ref{Const: Sigma, Lambda and Gamma} and Definition \ref{Def: subgroup H} respectively and $\delta \in \Aut(\Gamma)$. Assume that $g \in \Gamma$ and $H_{0} < H$ is a finite index subgroup such that $h\delta(H_{0})h^{-1} \subset H$. Then $\delta$ is inner. 
\end{lemma}

\begin{proof}
    By Corollary \ref{Corr: automorphisms of Gamma}, up to an inner automorphism, $\delta$ is of the form $\delta_{1} *_{\Sigma} (\delta_{1}|_{\Sigma} \times  \delta_{2})$ where $\delta_{1} \in \Aut(\Gamma_{1})$, $\delta_{2}\in \Aut(\Lambda)$ and $\delta_{1}(\Sigma) = \Sigma$. Since $A_{0} \in H$ and $H_{0}< H$ is a finite index subgroup, we can choose an integer $s > 0$ such that $A_{0}^{s} \in H_{0}$. Let $g = a_{0}b_{1}a_{1}\dotdot b_{n}a_{n}$ be a reduced expression for $g$. Then we have that:  
\begin{equation}
\label{eq delta2IsId.1}
    g\delta(A_{0}^{s})g^{-1} = a_{0}b_{1}a_{1}\dotdot a_{n-1}b_{n}(a_{n}\delta_{2}(A_{0}^{s})a_{n}^{-1})b_{n}^{-1}a_{n-1}^{-1}\dotdot b_{1}^{-1}a_{0}^{-1}
\end{equation}
is a reduced expression for $g\delta(A_{0}^{s})g^{-1}$ and belongs to $H$. Note that reduced expressions in $H$ are precisely the words of the form $b_{1}b_{2}\dotdot b_{k}$ where each $b_{i}$ is a block containing either a letter $A_{0}^{m}$ for some integer $m \neq 0$ or a block containing three letters $g_{i}^{-1}A_{i}^{n}g_{i}$ for some integer $n \neq 0$. Thus $a_{n}\delta_{2}(A_{0}^{s})a_{n}^{-1}$ must be of the form $A_{i}^{r_{i}}$ for some non-zero integer $r_{i}$ and some $0 \leq i \leq k$. As in the proof of Lemma \ref{Prop: outer automorphism groups of linear groups}, since $A_{i}^{r_{i}}$ for $i \neq 0$ and $\delta_{2}(A_{0}^{s})$ have distinct eigenvalues, we have that $a_{n}\delta_{2}(A_{0}^{s})a_{n}^{-1}$ has to be of the form $A_{0}^{r_{0}}$. Hence equation \ref{eq delta2IsId.1} can be written as: 
\begin{equation}
\label{eq delta2IsId.2}
    g\delta(A_{0}^{s})g^{-1} = a_{0}b_{1}a_{1}\dotdot a_{n-1}b_{n}A_{0}^{r_{0}}b_{n}^{-1}a_{n-1}^{-1}\dotdot b_{1}^{-1}a_{0}^{-1}
\end{equation}
Now comparing equation \ref{eq delta2IsId.2} with a reduced expression in $H$ and using Remark \ref{Remark: reduced words uniqueness}, we see that there exists $\sigma \in \Sigma$ such that $a_{0}b_{1}a_{1}...a_{n-1}b_{n}\sigma \in H$. We can conclude that $g \in H(\Sigma \times \Lambda)$. Multiplying $g$ on the left with an element of $H$, we may assume that $g \in \Sigma \times \Lambda$. Now we write $g = a\sigma^{-1}$ with $a \in \Lambda$ and $\sigma \in \Sigma$. For all $i \in \{1,...,k\}$, since $g_{i}^{-1}A_{i}g_{i} \in H$, we can find positive integers $s_{i}> 0$ such that $g_{i}^{-1}A_{i}^{s_{i}}g_{i} \in H_{0}$. Thus we have that $a\sigma^{-1}\delta_{1}(g_{i})^{-1} \delta_{2}(A_{i}^{s_{i}})\delta_{1}(g_{i})\sigma a^{-1} \in H$. Note that $a(\sigma^{-1}\delta_{1}(g_{i}^{-1})\delta_{2}(A_{i}^{s_{i}})(\delta_{1}(g_{i})\sigma)a^{-1}$ is a reduced expression. Once again we compare this to reduced expressions in $H$ and conclude that either $a = e$ or $a$ is a nonzero power of $A_{0}$. Thus we have that $a \in H$. Multiplying $g$ on the left with $a^{-1}$ we may assume that $g = \sigma^{-1} \in \Sigma$. Since for all $i \in \{1,...,k\}$, the elements $g_{i}^{-1}A_{i}g_{i}A_{0}$ belong to $H$, a non zero power of these belong to $H_{0}$. Thus we have: 
\begin{equation}
\label{eq delta2IsId.3}
    (\sigma^{-1}\delta_{1}(g_{i})^{-1})\delta_{2}(A_{i})\delta_{1}(g_{i})\delta_{2}(A_{0})\sigma \in H
\end{equation}
Once again comparing equation \ref{eq delta2IsId.3} to a standard reduced expression in $H$, and looking at the eigenvalues, we get that $\delta_{2}(A_{i}) = A_{i}^{\pm 1}$ for all $i \in \{0,...,k\}$. Now by Lemma \ref{Lemma: no non trivial homomorphisms of lambda}, we have that $\delta_{2}|_{\Lambda} = \id$. 

Now we note that $g_{i}A_{i}^{s_{i}}g_{i} \in H_{0}$ for all $i \in \{1,...,k\}$ and we finally get that
\begin{equation}
\label{eq delta2IsId.4}
\sigma^{-1}\delta_{1}(g_{i})^{-1}A_{i}^{s_{i}}\delta_{1}(g_{i})\sigma \in H
\end{equation}
Once again from equation \ref{eq delta2IsId.4} and Remark \ref{Remark: reduced words uniqueness}, we find elements $\sigma_{i} \in \Sigma$ such that $\delta_{1}(g_{i})\sigma = \sigma_{i}g_{i}$ for all $i$. By Proposition \ref{Prop: Group homomorphisms preserving Sigma are Inner}, $\delta_{1} = \Ad(X)$ with $X \in \widetilde{\Sigma}$. Thus we have that $Xg_{i}X^{-1}\sigma = \sigma_{i}g_{i}$ and hence $X^{-1}\sigma = g_{i}^{-1}X^{-1}\sigma_{i}g_{i}$. Since $X^{-1}\sigma_{i} \in \widetilde{\Sigma}$, we have that $X^{-1}\sigma \in \bigcap_{i=1}^{k}g_{i}^{-1}\widetilde{\Sigma}g_{i}$. Thus by Lemma \ref{Lemma: Sigma is non-normal}, we have that $X^{-1}\sigma \in \{ \pm I\}$. Hence we have $X = \sigma$ or $X = - \sigma$ and in both cases, $\Ad(X) = \Ad(\sigma)$, thus concluding our proof.  
\end{proof}

\begin{lemma}
    \label{infiniteindex}
    If $g \in \Gamma \backslash H$, then $gHg^{-1} \cap H$ has infinite index in $H$. 
\end{lemma}
\begin{proof}
    Let $K = gHg^{-1} \cap H$ for some $g \in \Gamma \backslash H$ and suppose that $K$ has finite index in $H$. Then there exists a positive integer $s > 0$ such that $A_{0}^{s} \in K$ so that $g^{-1}A_{0}^{s}g \in H$. Suppose that $g$ has a reduced expression that begins with an element $a \in \Gamma_{1} \backslash \Sigma$. Recall that reduced expressions in $H$ consist of elements of the form $A_{0}^{s}$ of blocks of the form $g_{i}A_{i}^{r_{i}}g_{i}^{-1}$ for non zero integers $r_{i},s$ and for $1 \leq i \leq n$, exactly as in the proof of Lemma \ref{Lemma: delta2 is identity}. Then it is easy to see that the reduced expression of $g^{-1}A_{0}^{s}g$ is also of the form $B_{k}^{-1}B_{k-1}^{-1}..B_{1} A_{0}^{s} B_{1}B_{2}...B_{k}$ with $k > 0$, where $B_{j}$'s are such blocks and the first block $B_{1}$ is strictly of the form $g_{i}A_{i}^{r_{i}}g_{i}^{-1}$ for some $1 \leq i \leq n$ and a non zero integer $r_{i}$. One can then conclude that the reduced expression for $g$ is precisely of the form $B_{1}B_{2}...B_{k}$, and hence $g \in H$, which is a contradiction. 

    For the second, case consider now that $g$ has a reduced expression that begins with an element $\lambda \in \Lambda$. Then for some $1 \leq i \leq n$, we have that $g^{-1} (g_{i}A_{i}^{s}g_{i}^{-1}) g \in H$. Looking at its reduced expression in $H$, one can immediately conclude that $\lambda$ must be of the form $A_{0}^{s}$ for some non zero integer $s$ and that the reduced expression of $g$ must be of the form $B_{1}B_{2}....B_{k}$ with $k > 0$ and $B_{1}$ is of the form $A_{0}^s$. Thus once again this shows that $g \in H$, hence giving us a contradiction. Notice that the same argument holds if $g \in \Lambda$. For the last case, suppose that $g = \sigma\lambda \in \Sigma \times \Lambda$ then by the exact same argument we have that $\sigma \lambda$ must be equal to $A_{0}^{s}$ and hence belong to $H$. 
\end{proof}

\begin{lemma}
    \label{Lemma: infinite conjugacy class}
    Every element $g \in \Gamma \backslash \{\pm I\}$ has an infinite conjugacy class.   
\end{lemma}
\begin{proof}
    Recall that $\Gamma = \Gamma_{1} \ast_{\Sigma} (\Sigma \times \Lambda)$. It is well known that $\PSL(6,\mathbb{Z})$ is an icc group and hence for $g \in \Sigma$ such that $g \neq \{\pm I\}$, $g$ has an infinite conjugacy class in $\Gamma$. If $g \in (\Sigma \times \Lambda) \backslash \Sigma$, then we can pick infinitely many distinct elements $c_{k} \in \Gamma_{1} \backslash \Sigma$ and note that $c_{k}gc_{k}^{-1}$ are all distinct elements in $\Gamma$. Hence it has an infinite conjugacy class. Otherwise, consider a reduced expression $g = a_{0}b_{1}a_{1}...b_{n}a_{n}$ with $a_{i} \in \Lambda$ and $b_{i} \in \Gamma_{1} \backslash \Sigma$ as in Remark \ref{Remark: reduced words uniqueness}. If $a_{0}$ and $a_{n}$ are both nontrivial then we can pick infinitely many distinct $c_{k} \in \Gamma_{1} \backslash \Sigma$, and if they are both trivial we pick infinitely many distinct $c_{k} \in \Lambda$. If $a_{0} = e$ and $a_{n} \neq e$, we can pick infinitely many distinct elements $c_{k} \in \Lambda$ such that $c_{k} \neq a_{n}^{-1}$ for any $k$, and similarly for $a_{0} \neq e$ and $a_{n} = e$. In all these cases $c_{k}^{-1}gc_{k}$ are all distinct and hence $g$ has an infinite conjugacy class. 
\end{proof}

\subsection{Terminology on group actions}

Let $G \curvearrowright (W,\rho)$ be an action of a countable discrete group on a standard measure space. A Borel subset $E \subset W$ is called a \textit{wandering set} if $\rho(gE \cap E) = 0$ for all non-trivial $g \in G$. By \cite[Lemma 1.0.7 and Theorem 1.1.1]{Aaronson97}, we have the following proposition: 

\begin{proposition}
    \label{Prop: Dissipative part of action exists}
    Let $G \curvearrowright (W,\rho)$ be an essentially free countable discrete group action on a standard measure space. Then there exists a measurable subset $\cD(W) \subset W$ such that the following conditions hold: 
    \begin{enumerate}
        \item For every wandering set $E \subset W$, we have $\rho(E \backslash \cD(W)) = 0$,
        \item If $U \subset W$ is a non-null measurable subset, then there exists a non-null wandering set $E \subset U$. 
    \end{enumerate}
    Moreover the set $\cD(W)$ is unique up to measure zero, i.e. if $\cD'(W)$ satisfies the conditions above, then $\rho(\cD(W) \Delta \cD'(W)) = 0$. 
\end{proposition}

Let $G \curvearrowright (W,\rho)$ be an essentially free countable discrete group action. Then the measurable subset $\cD(W)$ is called the \textit{dissipative part} of the action. Notice now that for all $g \in G$, the subset $g \cdot \cD(W)$ satisfies the two conditions in Proposition \ref{Prop: Dissipative part of action exists} as well. Since $\cD(W)$ is unique, we have that $\rho(\cD(W) \Delta g \cdot \cD(W)) = 0$ for all $g \in G$. So, after possibly discarding a measure zero set, one can assume that $\cD(W)$ is $G$-invariant. Hence the complement $\cC(W) = W \backslash \cD(W)$ is also $G$-invariant and $G \curvearrowright \cC(W)$ is called the \textit{conservative part} of the action. Any nonsingular action hence decomposes into its dissipative and conservative parts. 

\begin{definition}
    \label{Def: dissipative and conservative actions}
    Let $G \curvearrowright (W,\rho)$ be a nonsingular action of a countable discrete group on a standard measure space. Then the action is called \textit{dissipative} if $\cD(W) = W$ up to measure zero and \textit{conservative} if $\cC(W) = W$ up to measure zero.
\end{definition}

Clearly if the action is furthermore ergodic, then it is either dissipative or conservative. Now we define the notion of a fundamental domain and state the following proposition that outlines the difference in behaviour of the conservative and dissipative parts. For a proof of Proposition \ref{Prop: difference in dissiptaive and conservative parts} below, we refer the reader to \cite[Propositions 1.6.1 and 1.6.2]{Aaronson97}. 

\begin{definition}
\label{Def: fundamental domain}
    Let $G \curvearrowright (W,\rho)$ be an essentially free action of a countable discrete group. A wandering set $F \subset W$ is called a \textit{fundamental domain} if $\bigcup_{g \in G}g \cdot F = W$ up to measure zero. 
\end{definition}

\begin{proposition}
    \label{Prop: difference in dissiptaive and conservative parts}
    Let $G \curvearrowright (W,\rho)$ be an essentially free action of a countable discrete group. Then the action on the dissipative part $G \curvearrowright \cD(W)$ admits a fundamental domain. However the action on the conservative part $G \curvearrowright \cC(W)$ is \textit{recurrent}, i.e., for every non-null measurable subset $E \subset \cC(W)$, there are infinitely many $g \in G$ such that $\rho(g \cdot E \cap E) > 0$.  
\end{proposition}

Notice that from Proposition \ref{Prop: difference in dissiptaive and conservative parts}, it follows immediately that a probability measure preserving action $G \curvearrowright (W,\rho)$ is always conservative. We now state the following result, essentially due to Schmidt and Walters, \cite[Theorem 2.3]{MR0675419}. As in \cite[Theorem 7.3]{AIM21}, we note that the recurrence assumption is weaker than the proper ergodicity assumption in \cite{AIM21} but the same proof holds for recurrent actions. We also remark that this result holds as in the references, more generally for locally compact group actions, but we do not need it for this paper. 

\begin{theorem}\textup{(c.f. \cite[Theorem 7.3]{AIM21})}
    \label{Thm: Recurrent and mixing means ergodic}
    Let $G$ be a countable discrete group. Let $G \curvearrowright (W,\rho)$ be a recurrent nonsingular action and $G \curvearrowright (W',\rho')$ be a p.m.p. mixing action. Then we have: 
    \begin{align*}
        L^{\infty}(W \times W')^{G} = L^{\infty}(W)^{G} \otimes 1
    \end{align*}
\end{theorem}

In the next sections we shall often use the notion of induced actions. We define the notion here and give some examples: 

\begin{definition}
    \label{Def: induced actions}
    Let $G \curvearrowright (W,\rho)$ be a nonsingular action of a countable discrete group. Then $G \curvearrowright (W,\rho)$ is said to be \textit{induced} from $H \curvearrowright (W_{0},\rho)$ for a non-null Borel subset $W_{0} \subset W$  and subgroup $H < G$ if $W_{0}$ is $H$-invariant and $\{g \cdot W_{0} \; | \; g \in G/ H\}$ forms a Borel partition of $W$. Sometimes, depending on the context, we will say that $G \curvearrowright (W,\rho)$ is induced from $H$, which means there exists such a Borel subset $(W_{0},\rho)$. Equivalently $G \curvearrowright (W,\rho)$ is induced from $H$ if there is a $G$-equivariant Borel map $W \rightarrow G/ H$.     
\end{definition}

Notice that every action is trivially induced from itself. If $G \curvearrowright (W,\rho)$ admits a fundamental domain $E$, then it is induced from $\{e\} \curvearrowright (E,\rho)$. Moreover, if $H < G$ is a subgroup, then the translation action $G \curvearrowright G$ is induced from $H \curvearrowright H$. Now we introduce the notion of mixing and weakly mixing actions, which are in some sense, stronger forms of ergodicity. 

\begin{definition}
    \label{Def: mixing actions}
    Let $G \curvearrowright (W,\rho)$ be a p.m.p. action of a discrete countable group on a standard measure space. The action is called \textit{mixing} if for all Borel subsets $A,B \subset W$, we have that: 
    \begin{align*}
        \lim_{g \rightarrow +\infty} \rho(gA \cap B) = \rho(A) \rho(B)
    \end{align*}
\end{definition}
It can be checked that for a mixing action $G \curvearrowright (W,\rho)$, if $H$ is an infinite subgroup of $G$, then the action $H \curvearrowright (W,\rho)$ is also mixing. 

\begin{definition}
    \label{Def: weakly mixing action}
    Let $G \curvearrowright (W,\rho)$ be a nonsingular action of a countable discrete group on a standard measure space. The action is called \textit{weakly mixing} if for every ergodic p.m.p. action $G \curvearrowright (W',\rho')$, the diagonal action $G \curvearrowright (W \times W', \rho \times \rho')$ given by $g \cdot (x,y) = (gx,gy)$ is again ergodic. 
\end{definition}

It is clear almost from the definitions that both mixing and weakly mixing imply ergodicity. One can also check (see \cite[Proposition 2.2.11]{PetersonLectureNotes11}) that a mixing action is weakly mixing. Another important property of group actions that we will use in the following sections is the notion of malleability and strong malleability of actions, introduced by Popa in a series of articles \cite{Popa06}, \cite{Popa06StrongRigidity1}, \cite{Popa06StrongRigidity2}, \cite{Popa07}. We only define such actions and discuss an example briefly. For more details, we refer the reader to \cite{Popa07}. 

\begin{definition}
    \label{Def: strongly malleable action}
    Let $G$ be a discrete countable group and let $G \curvearrowright (W, \rho)$ be a measure preserving action. The action is called \textit{s-malleable} if there exists a one-parameter group $(\alpha_{t})_{t \in \mathbb{R}}$ of measure preserving transformations in $\Aut(W \times W)$ and a nonsingular transformation $\beta = \beta^{-1}$ in $\Aut(W \times W)$ such that the following hold:
\begin{enumerate}
\item The maps $\alpha_{t}$ for all $t \in \mathbb{R}$ and $\beta$ commute with the diagonal action $G \curvearrowright W \times W$,

\item For a.e. $(x,y) \in W \times W$, we have $\alpha_{1}(x,y) \in \{y\} \times W$ and $\beta(x,y) \in \{x\} \times W$,  

\item For all $t \in \mathbb{R}$, we have $\alpha_{t} \circ \beta = \beta \circ \alpha_{-t}$.  
\end{enumerate}
\end{definition}

As we mentioned earlier, we shall use most of the notions introduced in this section in the context of countable discrete groups. A major source of examples of s-malleable actions come from generalized Bernoulli actions. We introduce these actions here, as in \cite[Example 4.4]{Popa07}. We shall use such generalized Bernoulli actions later and for a detailed treatment, we refer the reader to \cite{PopaVaes08}. 

We now look at a class of examples that are of central importance in this article. Let $(W_{0},\rho_{0})$ be a standard measure space and let $G \curvearrowright K$ be a countable discrete group acting on a countable set. Let $(W,\rho) = \Pi_{k \in K}(W_{0},\rho_{0})_{k}$ denote the countable product and let $\sigma: G \curvearrowright (W,\rho)$ be given by: 
    \begin{align*}
        \sigma(g)((x_{k})_{k}) = (x_{g^{-1}k})_{k}
    \end{align*}
Such actions are called \textit{generalized Bernoulli actions}. The space $(W_{0},\rho_{0})$ is often called the \textit{base space} of the action. When $K = G$ and the action $G \curvearrowright K$ is given by left translation, it is simply called a \textit{Bernoulli action}. It was shown in \cite{Popa06} and \cite{Popa06StrongRigidity1} that if the base space $(W_{0},\rho_{0})$ is non-atomic, then such Bernoulli actions are s-malleable. In \cite{Popa06StrongRigidity1} and \cite{Popa06StrongRigidity2}, it was shown that for a class of groups called $w$-rigid groups, Bernoulli actions are mixing. However in general, generalized Bernoulli actions as above do not have the mixing property and are only weakly mixing (see \cite[Proposition 2.3]{PopaVaes08}).

\subsection{Type III$_{1}$ actions and von Neumann algebras}

In this section we give a brief review of the theory of type III actions and factors. A lot of these results appear in \cite{Connes-Takesaki_1977} and for an elaborate treatment, we refer the reader to any standard text on type III von Neumann algebras (for example \cite{SunderAnInvitation}). Since we are interested in von Neumann algebras with separable preduals, we can always assume the existence of faithful normal states. Recall that a factor $M$ is said to be of type III if it has no nontrivial finite projections. Given a faithful normal state $\phi$ on $M$, one has from Tomita-Takesaki theory a one-parameter group of automorphisms $(\sigma^{\phi}_{t})_{t \in \mathbb{R}}$ of $M$ called the \textit{modular automorphism group}. The modular automorphism group does not depend on the choice of a faithful normal state up to a one-parameter group of unitaries (\cite[Theorem 3.1.1]{SunderAnInvitation}), thanks to the Connes 2-cocycle derivative. The crossed product $\widetilde{M}$ is called the \textit{continuous core} of $M$ and is hence isomorphic for any choice of faithful normal state. A type III factor is said to be of type III$_{1}$ if $\widetilde{M}$ is a factor. 

Recall that the group of all *-isomorphisms of $M$ is denoted by $\Aut(M)$ and is called the \textit{automorphism group of $M$}. With the topology of pointwise norm convergence in $M_{*}$,  $\Aut(M)$ is a Polish group (see \cite{Haagerup75} for more details). The subgroup of all inner automorphisms denoted by $\Inn(M)$ is not necessarily a closed subgroup of $\Aut(M)$. If $\Inn(M)$ is closed in $\Aut(M)$, the factor is called \textit{full}. In case of full factors, the quotient group $\Out(M)$ which is called the \textit{outer automorphism groups} is also a Polish group. Hence for full factors, we have a canonical Polish group homomorphism $\delta: \mathbb{R} \rightarrow \Out(M)$, again thanks to the Connes 2-cocycle theorem. One can check that if $M$ is a factor of type III$_{1}$, then the modular automorphisms are necessarily outer. As a consequence the modular homomorphism $\delta: \mathbb{R} \rightarrow \Out(M)$ is injective. For more details on this topic we refer the reader to \cite[Section V]{Connes74}. 

An important class of examples of von Neumann factors come from ergodic countable Borel equivalence relations (\cite{MR578656} and \cite{MR578730}). Such equivalence relations often arise as orbit equivalence relations of free ergodic actions of countable discrete groups. Now suppose that $G$ is a countable discrete group and $G \actson (W,\rho)$ is a free, ergodic (non-singular) action on a standard probability space and let $\cS$ be the orbit equivalence relation. Then the action is said to be of type III when there is no equivalent finite or infinite measure $\rho$ such that the action is $\rho$-preserving. As one would expect, in such a situation the crossed product von Neumann algebra $L^{\infty}(W,\rho) \rtimes G = L(\cS)$ is a type III factor.  

Given a type III free ergodic group action $G \actson (W,\rho)$, a \textit{1-cocycle} for the action with values in a group $K$ is a measurable map $c: G \times W \rightarrow K$ satisfying $c(h,gx)c(g,x) = c(hg,x)$ for a.e. $x \in W$ and all $g \in G$. Two such 1-cocycles are called \textit{cohomologous} if there is \textit{a coboundary}, i.e., a measurable map $f: W \rightarrow H$ such that $c(g,x) = f(x)d(g,x)f(x)^{-1}$ for a.e. $x \in W$ and all $g \in G$. One can similarly define the notion of 1-cocycles for countable Borel equivalence relations and for orbit equivalence relations, they coincide. The group of 1-cocycles is denoted by $Z^{1}(\cS, K)$ and the subgroup of 1-cocycle is denoted by $B^{1}(\cS,K)$. The quotient is called the \textit{1-cohomology} group and is denoted by $H^{1}(\cS,K)$. In this article we shall consider such 1-cohomology groups of equivalence relations with $K = \mathbb{T}$, the circle group. In \cite{Moore76}, it is shown that with respect to the topology of pointwise convergence in measure, $H^{1}(\cS,\mathbb{T})$ turns into a Polish group. 

Recall that for a countable Borel equivalence relation $\cS$ on $(W,\rho)$, there are two $\sigma$-finite Borel measures $\rho_{1}$ and $\rho_{2}$ on $\cS$ as follows: 
\begin{align*}
    \rho_{1}(E) &= \int_{W} \#\{y \in W \; | \; (x,y) \in E\} \; d\rho(x) \\
    \rho_{2}(E) &= \int_{W} \#\{y \in W \; | \; (y,x) \in E\} \; d\rho(x)
\end{align*}

By definition the equivalence relation $\cS$ is non-singular if one has the $\rho_{1}$ and $\rho_{2}$ are equivalent measures (have the same null sets). Then the 1-cocycle $D: \cS \rightarrow \mathbb{R}$ given by $(x,y) \mapsto \log \frac{d\rho_{1}}{d \rho_{2}}(x,y)$ is called \textit{the logarithm of the Radon-Nikodym 1-cocyle}. Now let $\gamma$ be the Borel measure on $\mathbb{R}$ given by $d \gamma(t) = e^{-t} \; d \lambda(t)$ where $\lambda$ is the usual Lebesgue measure. Then the \textit{Maharam extension} of $\cS$ is the countable Borel equivalence relation $c(\cS)$ on the standard probability space $(W \times \mathbb{R}, \rho \times \gamma)$ given by: 
\begin{align*}
    ((x,s),(y,t)) \in c(\cS) \iff (x,y) \in \cS \text{ and } s = t + D(x,y)
\end{align*}
The Maharam extension preserves the infinite measure $\rho \times \lambda$ which is equivalent to the probability measure $\rho \times \gamma$. An equivalence relation $\cS$ of type III is said to be of type III$_{1}$ if the Maharam extension $c(R)$ is ergodic. In this case the Mahram extension is an ergodic equivalence relation of type II$_\infty$. It can be shown 
(see \cite[Theorem 1.6.20]{lirias3869720}) that the associated von Neumann algebra $L(c(\cS))$ is isomorphic to the continuous core of $L(\cS)$. We note that all these notions can be defined similarly for free ergodic group actions, and in this article we will go back and forth between these notions for group actions and their associated equivalence relations. 

Similar to the context of von Neumann algebras, one can define the automorphism groups of equivalence relations (see \cite{ConnesKrieger77}). By $\Aut(\cS)$, we mean the Polish group of all automorphisms, i.e. all nonsingular measure space automorphisms $\theta \in \Aut(W)$ such that $(x,y) \in \cS \iff (\theta x, \theta y) \in \cS$. Such automorphisms are often also called \textit{orbit equivalences}. Every automorphism  $\theta \in \Aut(\cS)$ of a type III equivalence relation induces an automorphism $\widetilde{\theta} \in \Aut(c(\cS))$ of the Maharam extension given by: 
\begin{align*}
    \widetilde{\theta} (x , s) = (\theta x, \omega(x) + s) 
\end{align*}
where $\omega(x) = \frac{d \theta_{*} \rho}{ d \rho} (x)$ for the Radon Nikodym derivative. Notice that $\widetilde{\theta}$ preserves the infinite measure $\rho \times \lambda$ where $\lambda$ is the Lebesgue measure on $\mathbb{R}$. 

The subgroup of all automorphisms $\theta$ such that $(x, \theta x) \in \cS$ for a.e. $x \in W$ is called the \textit{full group} of $\cS$ and is denoted by $[\cS]$. If $[\cS]$ is closed in $\Aut(\cS)$, then the quotient $\Out(\cS)$ is a Polish topological group. It is easy to see that every such an automorphism of $\cS$ induces an automorphism of $L(\cS)$ preserving the Cartan subalgebra $L^{\infty}(W,\rho)$. Moreover every element $c \in H^{1}(\cS,\mathbb{T})$ gives rise to an automorphism in $\Aut(L(\cS))$, in fact $H^{1}(\cS,\mathbb{T})$ corresponds to the automorphisms of $\Aut(\cS)$ that fix every element of $L^{\infty}(W,\rho)$ pointwise.

We refer the reader to \cite{MR578730} for a detailed treatment of the interplay between $\Aut(\cS)$ and $\Aut(L(\cS))$ and we recall the main result that we will use from \cite[Theorems 3 and 4]{MR578730}. Let $\cS$ be as above and $A \subset M$ denote the Cartan inclusion $L^{\infty}(W,\rho) \subset L(\cS)$ and assume that $L(\cS)$ is a full factor. Let us denote by $\Out(A \subset M)$
the Polish group of automorphisms preserving the Cartan subalgebra $A$. Then $\Out(A \subset M)$ is isomorphic and homeomorphic as Polish groups to $\Out(\cS) \ltimes H^{1}(\cS,\mathbb{T})$, where the action is given by $\theta \cdot c (x,y) = c(\theta^{-1} x, \theta^{-1} y)$.

\section{Construction}
\label{construction}

\begin{definition}
\label{Const: II infinity action}
Let $\Gamma_{1}$, $\Sigma$,  $\Lambda$, $\Gamma$ and $\pi$ be as in Definition \ref{Const: Sigma, Lambda and Gamma}. Let $X_{0} = \{0,1\}$ and $\mu_{0}$ be a probability measure on $X_{0}$ with $\mu_{0}(\{0\}) \neq \mu_{0}(\{1\})$. Let $(X, \mu) = (X_{0}, \mu_{0})^{\Gamma/H}$ and $\Gamma \actson (X,\mu)$ be the generalized Bernoulli action where $H$ is the subgroup of $\Gamma$ constructed in Definition \ref{Def: subgroup H}. We shall usually denote by $\lambda$ the Lebesgue measure on any Euclidean space, in this case on $\mathbb{R}^{6}$. Let $Y = \mathbb{R}^{6}$ and let $\Gamma_{1} \curvearrowright (Y,\lambda)$ be the $\lambda$-preserving translation action. Now let $(Z, \eta) = (X \times Y, \mu \times \lambda)$ and let $\alpha$ be the measure preserving action of $\Gamma$ on $Z$ given by $g \cdot (x,y) = (gx, \pi(g)y)$. We let $\mathcal{R}$ be the type $\textrm{II}_{\infty}$ orbit equivalence relation of the action $\alpha$.  
\end{definition}

\begin{definition}
\label{Const: III_1 action}
Let $Y_{1} = \mathbb{R}^{6}/ \mathbb{R}^{*}_{+}$ and consider the translation action of $\SL(6,\mathbb{Z})$ on $Y_{1}$. As in the proof of Lemma \ref{Lemma: type III1 and maharam ergodic}, there is a standard probability measure $\nu$ on $Y_{1}$ such that the translation action of $\SL(6,\mathbb{Z})$ on $Y_{1}$ is nonsingular. Now we let $\alpha_{1}$ be the action of $\Gamma$ on $(Z_{1}, \eta_{1}) = (X \times Y_{1}, \mu \times \nu)$ given by $g \cdot (x,y) = (gx, \pi(g)y)$. We let $\mathcal{R}_{1}$ be the orbit equivalence relation of the action $\alpha_{1}$. 
\end{definition}

It is possible to identify the action $\alpha_{1}$ with a different action to make computations easier. This has been done in \cite[Section 5]{lirias3869720}, and we summarize the results in the following lemma. 

\begin{lemma}
\label{Lemma: type III1 and maharam ergodic}
    The equivalence relation $\mathcal{R}_{1}$ as in Definition \ref{Const: III_1 action} is ergodic, of type $\textrm{III}_{1}$ and the Maharam extension of $\mathcal{R}_{1}$ is isomorphic to $\mathcal{R}$ as in Definition \ref{Const: II infinity action}.  
\end{lemma}

\begin{proof}
    We first identify $\mathbb{R}^{*}_{+}$ with $\mathbb{R}$ and consider the action of $\SL(6, \mathbb{Z}) \times \mathbb{R}$ on $\mathbb{R}^{6}$ given by:
    \begin{equation}
        (g,t) \cdot y  = (e^{-t/6}g)y
    \end{equation}
    where the $e^{-t/6}g$ acts on a vector $y$ by matrix multiplication. Since the determinant of $g$ is 1, we have that $\det(e^{-t/6}g) = e^{-t}$. Hence the $\mathbb{R}$ action scales the measure and the $\SL(6,\mathbb{Z})$ action is measure-preserving. By \cite[Proposition 4.1.14]{lirias3869720}, the quotient $\mathbb{R}^{6}/ \mathbb{R}$ can be equipped with a measure with respect to which the action $\SL(6,\mathbb{Z}) \curvearrowright \mathbb{R}^{6}/ \mathbb{R}$ is nonsingular. Now we can identify the action $\SL(6,\mathbb{Z}) \curvearrowright \mathbb{R}^{6}/\mathbb{R}^{*}_{+}$  with $\SL(6,\mathbb{Z}) \curvearrowright \mathbb{R}^{6}/ \mathbb{R}$. Once again by \cite[Proposition 4.1.14] {lirias3869720}, the Maharam extension of this action can be identified with the translation action $\SL(6,\mathbb{Z}) \curvearrowright \mathbb{R}^{6}$. Since the generalized Bernoulli action in Definition \ref{Const: III_1 action} is p.m.p., we have that the Maharam extension of $\mathcal{R}_{1}$ is $\mathcal{R}$. In particular since $\mathcal{R}$ is ergodic by Lemma \ref{Lemma: four fold diagonal action ergodic}, we have that $\mathcal{R}_{1}$ is of type $\textrm{III}_{1}$. From the identification above, the ergodicity of $\mathcal{R}$ implies that also $\mathcal{R}_{1}$ is ergodic and that concludes the proof.  
\end{proof}

\section{$\mathcal{U}_{\fin}$-cocycle superrigidity}

To compute the outer automorphism group that we want, the first step is to show that the constructed action is cocycle superrigid, which we define now. For cocycle superrigidity results, one often considers the class of $\cU_{\fin}$ groups introduced by Popa. A Polish group $G$ is said to be \textit{of finite type} or \textit{a $ \cU_{\fin}$ group} if it arises as a closed subgroup of the unitary group of a II$_{1}$ factor with separable predual. For example every discrete countable group and every compact second countable group is $\cU_{\fin}$.  

\begin{definition}
\label{Def Cocycle Superrigid}
    A non-singular essentially free ergodic action on a standard measure space $G \curvearrowright (E, \rho)$ is called \textit{cocycle superrigid with target group $K$} if any 1-cocycle $\omega: G \times E \rightarrow K$ is cohomologous to a cocycle $(g,x) \mapsto \delta(g)$ for a.e. $x \in E$ where $\delta: G \rightarrow K$ is a group homomorphism. We say that such an action is $\mathcal{U}_{\fin}$-cocycle superrigid if it is cocycle superrigid with any $\cU_{\fin}$ target group $K$. 
\end{definition}

Over the years, a lot of cocycle superrigidity results have appeared in the literature. In \cite{Popa07}, Popa showed that all Bernoulli actions of property (T) groups are $\cU_{\fin}$-cocycle superrigid. In \cite{MR2783933}, Ioana showed that profinite actions of any property (T) group are cocycle superrigid with discrete countable target groups. In this article we deal with actions that are not probability measure preserving. Hence for such cocycle superrigidity results, the notion of property (T) for groups has to be replaced by Zimmer's notion of property (T) for actions. If the action of a property (T) group is probability measure preserving, then the action it is automatically property (T), but for nonsingular actions this is no longer the case. 

\begin{definition}
    \label{Def: 1 cocycles and invariant vectors}
    Let $G \actson (E,\rho)$ be a nonsingular free ergodic action of a discrete countable group and $\cS$ be the orbit equivalence relation. Let $\cK$ be a separable Hilbert space and $c : \cS \rightarrow \cU(\cK)$ be a 1-cocycle. A \textit{invariant unit vector for $c$} is a Borel map $\xi: E \rightarrow \cK$ satisfying $\xi(x) = c(x,y)\xi(y)$ for a.e. $(x,y) \in \cS$ and $\|\xi(x)\| = 1$ for a.e. $x \in E$. Similarly a \textit{sequence of almost invariant unit vectors for $c$} is a sequence of Borel maps $\xi_{n} : X \rightarrow \cU(\cK)$ satisfying: 
    \begin{align*}
       \| \xi_{n}(x) - c(x,y)\xi_{n}(y) \| \rightarrow 0 \text{ for a.e. } (x,y) \in \cS
    \end{align*}
    and $\|\xi_{n}(x)\| = 1$ for all $n$, for a.e. $x \in E$. We say that $\cS$ has \textit{property (T)} if every 1-cocycle admitting a sequence of almost invariant unit vectors admits an invariant unit vector. 
\end{definition}

\begin{lemma}
\label{Lemma: four fold diagonal action ergodic}
    Let $(Z, \eta) = (X \times Y, \mu \times \nu)$ and $\Gamma \curvearrowright (Z, \eta)$ be as in Definition \ref{Const: II infinity action}. Then the four-fold diagonal action $\Gamma \curvearrowright Z \times Z \times Z \times Z$ is ergodic.
\end{lemma}

\begin{proof}
    We first notice that for an element $g \in \Gamma$, the subgroup $g\Sigma g^{-1}$ does not intersect $H$ except at $\{\pm I\}$. Therefore the action $\Sigma \curvearrowright X$ is mixing. Since the diagonal product of mixing actions is still mixing, we have that $\Sigma \curvearrowright X \times X \times X \times X$ is still mixing. By \cite[Lemma 5.6]{PopaVaes11}, we have that $\Sigma \curvearrowright Y \times Y \times Y \times Y$ is ergodic, and hence it is a conservative nonsingular action. By Theorem \ref{Thm: Recurrent and mixing means ergodic}, we have that $\Sigma \curvearrowright Z \times Z \times Z \times Z$ is ergodic and hence the four-fold diagonal action of $\Gamma$ is also ergodic.  
\end{proof}

\begin{lemma}
\label{Lemma: action is cocycle superrigid}
    The action $\Gamma \curvearrowright (Z, \eta) = (X \times Y, \mu \times \lambda)$ as in Definition \ref{Const: II infinity action} is $\mathcal{U}_{\fin}$-cocycle superrigid. 
\end{lemma}

\begin{proof}
Let $G$ be a group in Popa's class $\mathcal{U}_{\fin}$ and let $\omega: \Gamma \times (X \times Y) \rightarrow G$ be a 1-cocycle. We first consider the restriction of $\omega$ to $\SL(6,\mathbb{Z}) \times (X \times Y)$ and show that the restriction is cohomologous to a group homomorphism, essentially using \cite[Theorem 5.3]{PopaVaes11}. We know that generalized Bernoulli actions are s-malleable and as in the proof of \cite[Theorem 1.3]{PopaVaes11}, $\SL(6,\mathbb{Z}) \curvearrowright \mathbb{R}^{6}$ is s-malleable. Hence the action $\SL(6,\mathbb{Z}) \curvearrowright X \times Y$ is s-malleable. We claim that the diagonal action $\Gamma \curvearrowright X \times Y \times X \times Y$ has property (T). 

We show this in the same way as in Step 1 of the proof of \cite[Theorem 21]{MR3048005}. Consider the action $\Gamma_1 \curvearrowright \mathbb{R}^{6}$. Denoting by $e_{1},e_{2},...,e_{6}$ the basis vectors of $\mathbb{R}^{6}$, notice that the orbit of $(e_{1},e_{2})$ in the diagonal action $\Gamma_1 \curvearrowright Y \times Y$ has a measure zero complement. Let us denote by $H_{0}$ the stabilizer of the point $(e_{1},e_{2})$ under the $\Gamma_1$ action. Then we can identify the action $\Gamma_1 \curvearrowright (X \times Y \times X \times Y)$ with the action $\Gamma_1 \curvearrowright \SL(6,\mathbb{R})/H_{0} \times X \times X$. By \cite[Proposition 3.3]{PopaVaes11}, $\Gamma_1 \curvearrowright X \times Y \times X \times Y$ has property (T) if and only if $H_{0} \curvearrowright \SL(6,\mathbb{R}) / \Gamma_1 \times X \times X$ has property (T). Since this is a p.m.p free ergodic action, this has property (T) as $H_{0}$ has property (T) by \cite[Proposition 3.1]{PopaVaes11}. By Lemma \ref{Lemma: four fold diagonal action ergodic}, all the hypotheses of \cite[Theorem 5.3]{PopaVaes11} are satisfied and using the theorem, we have that $\omega|_{\Gamma_1}$ is cohomologous to a group homomorphism.   

As in the proof of Lemma \ref{Lemma: four fold diagonal action ergodic}, the diagonal action $\Sigma \curvearrowright X \times Y \times X \times Y$ is ergodic. Since $\Sigma$ commutes with $\Lambda$, we have that $\omega$ is a group homomorphism on $\Lambda$ by \cite[Lemma 5.5]{PopaVaes11}. Since $\Gamma$ is generated by $\Gamma_1$ and $\Lambda$, we have that $\omega$ is given by a group homomorphism on $\Gamma$, and this concludes the proof. 
\end{proof}

In the final proposition of this section, we shall compute the cohomology group $H^{1}(\mathcal{R}_{1}, \mathbb{T})$, i.e., the group of 1-cocycles of $\mathcal{R}_{1}$ with values in the circle up to coboundaries. 

\begin{proposition} 
\label{Proposition: calculation of 1 cocycle group}
Let $\mathcal{R}_{1}$ be as in Definition \ref{Const: III_1 action}. Then $H^{1}(\mathcal{R}_{1},\mathbb{T}) = \mathbb{R}$ 
\end{proposition}
\begin{proof}
    It is a known fact that $\Gamma_{1} = \SL(6,\mathbb{Z})$ and $\Lambda = \SL(3,\mathbb{Z})$ are both equal to their commutator subgroups and are hence perfect. Hence we have that $\Gamma$ is also perfect and there are no nontrivial group homomorphism from $\Gamma$ to $\mathbb{T}$. Now suppose $\Omega: \Gamma \times Z_{1} \rightarrow \mathbb{T}$ is a 1-cocycle. Since by Lemma \ref{Lemma: four fold diagonal action ergodic}, the diagonal action $\Gamma \curvearrowright Z_{1} \times Z_{1}$ is ergodic and by Lemma \ref{Lemma: action is cocycle superrigid}, the Maharam extension $\mathcal{R}$ is cocycle superrigid with target $\mathbb{T}$, we note that all the hypotheses of \cite[Theorem 5.3.1]{lirias3869720} are satisfied. Now by \cite[Theorem 5.3.1]{lirias3869720}, there exists a group homomorphism $\gamma: \mathbb{R} \rightarrow \mathbb{T}$ such that $\Omega$ is  cohomologous to the 1-cocycle: 
    \begin{equation}
       \Gamma \times Z_{1} \ni (g,z) \mapsto \gamma(-\omega(g,z)) \in \mathbb{T}
    \end{equation}
    where $\omega$ is the logarithm of the Radon Nikodym 1-cocycle. Since the group of homomorphisms from $\mathbb{R}$ to $\mathbb{T}$ is $\mathbb{R}$ and since each such homomorphism defines a 1-cocycle as above, we conclude that $H^{1}(\mathcal{R}_{1},\mathbb{T}) = \mathbb{R}$.
\end{proof}

\section{Computation of the outer automorphism group}
We begin this section with the following lemma to show that the action restricted to non-trivial normal subgroups of $\Gamma$ cannot be dissipative. 

\begin{lemma}
\label{NoNontrivialDissipativeActions}
    If $N \triangleleft \Gamma$ is a normal subgroup and $N \curvearrowright \mathbb{R}^{6} \times (X_{0},\mu_{0})^{\Gamma/H}$ is dissipative, then $N = \{I\}$ or $N = \{\pm I\}$. 
\end{lemma}
\begin{proof}
    Consider the normal subgroup $N_{0} = N \cap \ker(\pi)$. Thus $N_{0}$ acts only on the second component $N_{0} \curvearrowright (X,\mu) = (X_{0}, \mu_{0})^{\Gamma/H}$. Any essentially free pmp action of an infinite group is recurrent. We can choose a probability measure on $\mathbb{R}^{6}$ which is equivalent to the Lebesgue measure and note that the action $N_{0} \curvearrowright \mathbb{R}^{6} \times (X_{0},\mu_{0})^{\Gamma/H}$ is recurrent if $N_{0}$ is infinite. Since the action is dissipative, this gives a contradiction and we get that $N_{0}$ is finite. By Lemma \ref{Lemma: infinite conjugacy class}, we have that $N_{0} \subseteq \{\pm I\}$. However since $\pi(-I) = -I$, we have that $N_{0} = \{I\}$. Hence $\pi|_{N}: N \rightarrow \SL(6,\mathbb{Z})$ is injective. For an element $h \in \SL(3,\mathbb{Z})$ and $k \in N$, we have that $\pi(hkh^{-1}k^{-1}) = 0$ and hence by injectivity we have $hk = kh$. Thus $N$ commutes with $\SL(3,\mathbb{Z})$. Therefore we have that $N \subset \Sigma \subset \SL(6,\mathbb{Z})$. By normality we have that for all $i$, $g_{i}^{-1}Ng_{i} = N$. Hence by Lemma \ref{Lemma: Sigma is non-normal} we have: 
    \begin{equation}
        N = \bigcap_{i}g_{i}^{-1}Ng_{i} \subseteq \bigcap_{i}g_{i}^{-1}\widetilde{\Sigma}g_{i} \subseteq \{\pm I\} 
    \end{equation}
    and that concludes the proof. 
\end{proof}

Cocycle superrigidity often leads to orbit equivalence superrigidity results. The first step in this direction in the p.m.p setting was taken by Zimmer in \cite[Proposition 2.4]{Zimmer20} and then generalized by Popa in \cite[Theorem 5.6]{MR2342637}. The most general result in the setting of nonsingular actions of locally compact groups was proven by Popa and Vaes in \cite[Lemma 5.10]{PopaVaes11}. Such results are usually proven by applying cocycle superrigidity to the so called Zimmer cocycle of an automorphism. We describe this notion here and then prove the following simplified version of cocycle superrigidity $\implies$ OE-superrigidity result which is essentially proven in \cite[Lemma 2.4]{DrimbeVaes23}. 

\begin{definition}
    \label{Def: Zimmer cocycle}
    Let $G \actson (E,\rho)$ and $K \actson (F, \sigma)$ be essentially free ergodic nonsingular actions of discrete countable groups on standard measure spaces, and let $\Delta: E \rightarrow F$ be an orbit equivalence ($\Delta(G \cdot x) = K \cdot \Delta(x)$ for all $x \in E$. Then the \textit{Zimmer cocycle} of $\Delta$, denoted by $Z_{\Delta}$ is the measurable map $Z_{\Delta}: G \times E \rightarrow K$ given by: 
    \begin{align*}
        (g,x) \mapsto k \text{ where } \Delta(gx) = k \Delta(x)
    \end{align*}
    One can check easily that this indeed defines a 1 cocycle.
\end{definition}

\begin{lemma}
\label{Lemma: cocycle SR to OE SR}
    Let $G \curvearrowright (E,\rho)$ and $K \curvearrowright (F, \sigma)$ be nonsingular essentially free ergodic actions and $\Delta: E \rightarrow F$ be an orbit equivalence. Suppose the associated Zimmer 1-cocycle is cohomologous to a group homomorphism $G \rightarrow K$. Then there exists: 
    \begin{enumerate}
        \item A normal subgroup $N \triangleleft G$ such that the action $N \curvearrowright E$ admits a fundamental domain. 
        \item A subgroup $K_{0} < K$ and a non-null subset $F_{0} \subset F$ such that $F_{0}$ is $K_{0}$-invariant and $K \curvearrowright F$ is induced from $K_{0} \curvearrowright F_{0}$, as in Definition \ref{Def: induced actions}.  
        \item A nonsingular isomorphism $\Delta_{0}: E/N \rightarrow F_{0}$ and a group isomorphism $\delta: G/N \rightarrow G_{0}$ such that $\Delta_{0}(gx) = \delta(g) \cdot \Delta_{0}(x)$ for all $g \in G/N$ and a.e. $x \in E$ and $\Delta(x) \in K \cdot \Delta_{0}(N \cdot x)$ for a.e. $x \in E$. 
    \end{enumerate}
\end{lemma}

Notice that in point 3 of Lemma \ref{Lemma: cocycle SR to OE SR}, the map $\Delta_{0}$ has domain $E/N$, which is a standard Borel space. This is because $N \actson E$ admits a fundamental domain and hence the associated equivalence relation is of type I and therefore the orbit space (denoted as the quotient) is a standard Borel space. For more details on type I equivalence relations and smoothness properties, we refer the reader to \cite[Chapter 2]{CalderoniLectureNotes09}. 

\begin{proposition}
\label{Prop: every orbit equivalence is a conjugacy}
    Let $\Gamma \curvearrowright (Z, \eta)= (X \times Y, \mu \times \lambda)$ be as in Definition \ref{Const: II infinity action} and let $\mathcal{R}$ be the orbit equivalence relation. Then for any orbit equivalence $\Delta \in \Aut(\mathcal{R})$, up to multiplication by an inner automorphism of $\mathcal{R}$, we have that $\Delta(g \cdot z) = g \cdot \Delta(z)$
    for a.e. $z \in Z$ and for all $g \in \Gamma$. 
    
\end{proposition}
\begin{proof}
    Let $\Delta \in \Aut(\mathcal{R})$ be an orbit equivalence for the action. By Lemma \ref{Lemma: action is cocycle superrigid}, the associated Zimmer 1-cocycle is cohomologous to a group homomorphism. Hence we can apply Lemma \ref{Lemma: cocycle SR to OE SR} to find a normal subgroup $N$, a subgroup $\Gamma_{0}$ and a $\Gamma_{0}$-invariant subset $Z_{0} \subset X \times Y$ such that $\Gamma \curvearrowright X \times Y$ is induced from $\Gamma_{0} \curvearrowright Z_{0}$. Now by Lemma \ref{Lemma: four fold diagonal action ergodic}, the action is doubly ergodic, and hence by \cite[Lemma 6.1]{PopaVaes11}, $\Gamma \curvearrowright X \times Y$ is not nontrivially induced, and we can assume $\Gamma_{0} = \Gamma$ and $Z_{0} = X \times Y$ up to measure zero. Since $N \curvearrowright X \times Y$ admits a fundamental domain, it is dissipative. By Lemma \ref{NoNontrivialDissipativeActions}, $N$ is either $\{I\}$ or $\{\pm I\}$. We show next that $N$ can never be $\{\pm I\}$. 

    Suppose $N = \{\pm I\}$ and then as in Lemma \ref{Lemma: cocycle SR to OE SR} there must be an isomorphism $\delta: \Gamma \rightarrow \Gamma/N$. Now it can be checked that the center of $\Gamma$ is $\{\pm I\}$. Indeed notice that any element in the center of $\Gamma$ must be in $\Sigma$ and must commute with $\SL(6,\mathbb{Z})$, however it is a known fact that the center of $\SL(6,\mathbb{Z})$ is $\{\pm I\}$ and hence the center of $\Gamma$ is $\{\pm 1\}$. However the center of $\Gamma/ N$ is then clearly trivial, hence giving us a contradiction. Hence $N = \{ I \}$ and after composing $\delta$ with an inner automorphism, we have by Lemma \ref{Lemma: automorphisms of amalgamated free products} that $\delta: \Gamma \rightarrow \Gamma$ is of the form $\delta_{1} *_{\Sigma} (\delta_{1}|_{\Sigma} \times \delta_{2})$ where $\delta_{1} \in \Aut(\Gamma_{1})$ and $\delta_{2} \in \Aut(\Lambda)$. By point 3 of Lemma \ref{Lemma: cocycle SR to OE SR}, we also have now that there is an orbit equivalence $\Delta_{0} \in \Aut(\mathcal{R})$ such that $\Delta_{0} = \phi \circ \Delta$ for some element $\phi \in [\mathcal{R}]$ with the property that $\Delta_{0}(g \cdot z) = \delta(g)\cdot \Delta_{0}(z)$ for a.e. $z \in Z$ and all $g \in \Gamma$. 

    Let us denote by $H_{1} < H$ the subgroup of index 2 generated by $g_{i}^{-1}A_{i}g_{i}$ for all $i \in \{1,...,k\}$ and $A_{0}$. Clearly $H_{1} \subset \ker(\pi)$ and since $\pi \circ \delta = \delta_{1} \circ \pi$, we also have that $\delta(H_{1}) \subset \ker(\pi)$. By Lemma \ref{infiniteindex}, given any $g \in \Gamma \backslash H$, we have that $K = gHg^{-1} \cap H$ has infinite index in $H$. Let $h_{1},h_{2},h_{3},...$ be a sequence of elements in $H$ such that $h_{i}K \neq h_{j}K$ for $i \neq j$. Then one can check that for the action of $H$ on $\Gamma/H - \{eH\}$ the cosets $h_{i}g H$ are all distinct, thus proving that $H$ acts on $\Gamma/H - \{H\}$ with infinite orbits. Since $H_{1}$ is an index 2 subgroup, it is also straightforward to check that the action of $H_{1}$ on $\Gamma/H - \{eH\}$ has infinite orbits. By \cite[Proposition 2.3]{PopaVaes08}, the generalized Bernoulli action $H_{1} \curvearrowright \Pi_{\Gamma/H - \{eH\}} (X_{0},\mu_{0})$ is ergodic. Consider now an $H_{1}$-invariant function $f \in L^{
    \infty}(X \times Y)$. Since $H_{1} \in \ker(\pi)$, we have that $f(hx,y) = f(x,y)$ for a.e. $(x,y) \in X \times Y$ and all $h \in H_{1}$. By taking $f_{y}(x) = f(x,y)$, we get by ergodicity that $f_{y}$ is a function on $X_{0}^{eH}$ by which we mean the copy of $X_{0}$ in the fiber over $eH$. To summarize we have:  
    \begin{equation}
    \label{eq aa}
        L^{\infty}(X \times Y)^{H_{1}} = L^{\infty}(X_{0}^{eH} \times Y)
    \end{equation}
    \textit{Claim 1: there exists $h \in \Gamma$ and a finite index subgroup $H_{0} < H_{1}$ such that $h \delta(H_{0})h^{-1} \subset H$.} Suppose that Claim 1 is false, then we will show that:
    \begin{equation}
    \label{eq ab}
    L^{\infty}(X \times Y)^{\delta(H_{1})} \subseteq 1 \otimes L^{\infty}(Y)
    \end{equation}
    Suppose Equation \ref{eq ab} does not hold, then we have a non constant function $f \in L^{\infty}(X)^{\delta(H_{1})}$. This means that $\delta(H_{1}) \actson X$ is not ergodic and by \cite[Proposition 2.3]{PopaVaes08}, the action $\delta(H_{1})$ acts on $\Gamma/H$ has at least one finite orbit. Hence there exists $gH \in \Gamma/H$ such that $\{kgH \; | \;  k \in \delta(H_{1})\}$ is a finite set. Let $K_{0}< \delta(H_{1})$ be the finite index subgroup $K_{0} = \{k \in \delta(H_{1}) \; | \; k \cdot gH = gH\}$. Taking $H_{0} = \delta^{-1}(K_{0})$, we have that $H_{0}$ is a finite index subgroup of $H_{1}$, and that $g\delta(H_{0})g^{-1} \subset H$, contradicting that 'Claim 1' is false. Hence assuming Claim 1, we have that Equation \ref{eq ab} is true. Now from Equations \ref{eq aa} and \ref{eq ab}, one can calculate: 
    \begin{align*} 
    \Delta_{0}(L^{\infty}(X_{0}^{eH} \times Y)) &= \Delta_{0}(L^{\infty}(X \times Y)^{H_{1}}) \\ = L^{\infty}(X \times Y)^{\delta(H_{1})} &\subseteq 1 \otimes L^{\infty}(Y)
    \end{align*}
    Since $1 \otimes L^{\infty}(Y)$ is globally $\Gamma$-invariant, it follows that: 
    \begin{equation}
    \Delta_{0}(L^{\infty}(X \times Y)) \subseteq 1 \otimes L^{\infty}(Y)
    \end{equation}
    which is absurd and thus we have proved Claim 1. Now by Lemma \ref{Lemma: delta2 is identity}, we have that $\delta$ is inner. Suppose $\delta(g) = kgk^{-1}$ for some $k \in \Gamma$. Then by replacing $\Delta$ by $k^{-1}\Delta$, we have that $\Delta(g \cdot z) = g \cdot \Delta(z)$ for all $g \in \Gamma$ and a.e. $z \in Z$
\end{proof}

Recall that an automorphism of an equivalence relation $\cR$ is called measure preserving if it preserves (a possibly infinite $\sigma$-finite) measure on the standard Borel space. Recall that every automorphism $\theta \in \Aut(\cR_{1})$ induces such a measure preserving automorphism of its Maharam extension.

\begin{proposition}
\label{Prop: every conjugacy is trivial}
    Let $\Gamma \curvearrowright (Z, \eta) = (X \times Y, \mu \times \lambda)$ be as in Definition \ref{Const: II infinity action}. Let $\Delta \in \Aut(\mathcal{R})$ be a measure preserving orbit equivalence satisfying $\Delta(g \cdot z) = g \cdot \Delta(z)$ for a.e. $z \in Z$ and for all $g \in \Gamma$. Then $\Delta = (\id \times \rho_{s})$ for some $s \in \mathbb{R} \backslash \{0\}$ where $\rho_{s}(y) = sy$.  
\end{proposition}

\begin{proof}
    We show that for an orbit equivalence $\Delta$ satisfying the hypotheses of the proposition, $\Delta = \id \times \rho_{s}$ in the following three steps. 
    
    \textit{Claim 1}: $\Delta(1 \otimes L^{\infty}(Y)) = 1 \otimes L^{\infty}(Y)$. 
    Since $\ker(\pi)$ acts on $\Gamma/H$ with infinite orbits, we have that the generalized Bernoulli action of $\ker(\pi)$ on $X$ is ergodic. Now let $1 \otimes f \in 1 \otimes L^{\infty}(Y)$ and consider the function $g = \Delta(1 \otimes f)$. We have that $g(hx,y) = g(x,y)$ for all $h \in \ker(\pi)$. For each $y \in Y$, consider the functions $g_{y}(x) = g(x,y)$. By ergodicity we have that $g_{y}$ is constant almost everywhere. Thus $g \in 1 \otimes L^{\infty}(Y)$.  
    
\textit{Claim 2}: After composing $\Delta$ with $(\id \times \rho_{s})$ for $s \in \mathbb{R} \backslash \{0\}$ if needed, we have that $\Delta(1 \otimes f) = 1 \otimes f$ 
for all $f \in L^{\infty}(Y)$. To check this, we first note that by Claim 1, the restriction of $\Delta$ to $L^{\infty}(Y)$ commutes with the action $\SL(6, \mathbb{Z}) \curvearrowright \mathbb{R}^{6}$. By \cite[Theorem D]{85852f12-3669-343b-b6ad-176dfe8edb82}, then for a.e. $y \in Y$, $\Delta(y) = g \cdot y$ for some $g \in \GL(6,\mathbb{R})$ that commutes with all elements in $\SL(6,\mathbb{Z})$. Thus $g$ is of the form $s \cdot I$ and $\Delta$ restricted to $L^{\infty}(Y)$ is of the form $\rho_{s}$ where $s \in \mathbb{R} \backslash \{0\}$ as required. 

\textit{Claim 3}: $\Delta = (\id \times \rho_{s})$ for some $s \in \mathbb{R} / \{0\}$. To check this, first let us write $\Delta(x,y) = (\Delta_{1}(x,y),\Delta_{2}(x,y))$. Suppose that as in Claim 2, $\widetilde{\Delta} = (\id \times \rho_{s^{-1}}) \circ \Delta$ restricted to $L^{\infty}(Y)$ is the identity. One can check as in Equation \ref{eq aa} from the proof of Proposition \ref{Prop: every orbit equivalence is a conjugacy} that $\widetilde{\Delta}(L^{\infty}(X_{0}^{eH} \times Y)) = L^{\infty}(X_{0}^{eH} \times Y)$. As a consequence we have that for a.e. $y \in Y$ the map given by $\widetilde{\Delta}_{y}(x) = \widetilde{\Delta}_{1}(x,y)$ is an automorphism of $L^{\infty}(X_{0},\mu_{0})$. Since $\widetilde{\Delta}$ is the product of a measure preserving and a measure scaling automorphism, $\widetilde{\Delta}$ must be measure scaling. Since $\widetilde{\Delta}|_{L^{\infty}(Y)} = \id|_{L^{\infty}(Y)}$, we must have that $\widetilde{\Delta}$ is measure preserving. Hence for a.e. $y \in Y$, we have that $\widetilde{\Delta}_{y}$ is measure preserving, and hence for a.e. $y \in Y$, $\widetilde{\Delta}_{y} = \id$ as $\mu_{0}(0) \neq \mu_{0}(1)$ thus proving the claim.
\end{proof}
To prove the uniqueness of Cartan subalgebras we use a general result for certain actions of amalgamated free product groups as in \cite{Vaes14}. As observed in \cite[Remark 8.3]{Vaes14}, the same proof as \cite[Theorem 8.1]{Vaes14} works for the following slightly more general result about uniqueness of Cartan subalgebras in this setting: 

\begin{theorem}[c.f. Theorem 8.1 and Remark 8.3 in \cite{Vaes14}]
\label{Thm: on uniqueness of Cartan from Vaes}
    Let $G = G_{1} *_{S} G_{2}$ be any amalgamated free product group and assume there exists $g_{1},...,g_{n} \in G$ such that $\cap_{i=1}^{n} g_{i}S g_{i}^{-1}$ is finite. Let $G \curvearrowright (W_{1}, \rho_{1})$ be any nonsingular free ergodic action with orbit equivalence relation $\cQ_{1}$. Let $G \curvearrowright (W, \rho)$ denote the infinite measure preserving Maharam extension with orbit equivalence relation $\cQ$. For $i = 1,2$, let $L_{i} < G_{i}$ be subgroups satisfying the following condition:
    \begin{itemize}
        \item [($\ast$)] For every Borel set $U \subset W$ with $0 <\rho(U) < \infty$, almost every $\cQ(L_{i} \curvearrowright W)|_{U}$ equivalence class consists of infinitely many $\cQ(S \cap L_{i} \curvearrowright W)|_{U}$ equivalence classes.
    \end{itemize} 
    Then $L^{\infty}(W_{1},\rho_{1})$ is the unique Cartan subalgebra up to unitary conjugacy in $L(\cQ_{1}) = L^{\infty}(W_{1},\rho_{1}) \rtimes G$.      
\end{theorem}
\begin{proof}
    Let $A = L^{\infty}(W,\rho)$ and let $\tau$ be the canonical semifinite trace on $M = A \rtimes G$. Consider a projection $p \in A$ such that $0< \tau(p) < \infty$. As in the proof of \cite[Theorem 8.1]{Vaes14}, for $i \in \{1,2\}$, it is enough to find unitaries $u_{i},v_{i} \in p(A \times G_{i})p$ that satisfies: 
    \begin{equation}
    \label{existence of unitaries}
        E_{p(A \rtimes S)p}(u_{i}) = E_{p(A \rtimes S)p}(v_{i}) = E_{p(A \rtimes S)p}(u_{i}^{*}v_{i}) = 0
    \end{equation}
    where $E_{B}$ for a subalgebra $B \subset M$ denotes the unique faithful normal conditional expectation onto $B$, which exists if $\tau|_{B}$ is still semifinite. As observed in \cite[Remark 8.3]{Vaes14}, by \cite[Lemma 2.6]{IKT09} we can find such unitaries $u_{i},v_{i}$ in $p(A \rtimes L_{i})p$ such that: 
    \begin{equation}
        E_{p(A \rtimes S \cap L_{i})p}(u_{i}) = E_{p(A \rtimes S \cap L_{i})p}(v_{i}) = E_{p(A \rtimes S \cap L_{i})p}(u_{i}^{*}v_{i}) = 0
    \end{equation}
    Now notice that $E_{p(A \rtimes S)p}(x) = E_{p(A \rtimes S \cap L_{i})p}(x)$ for all $x \in p(B \rtimes L_{i})p$ and hence the unitaries $u_{i},v_{i}$ satisfy equation \ref{existence of unitaries}, as required.
\end{proof}

Suppose as in Theorem \ref{Thm: on uniqueness of Cartan from Vaes}, we have subgroups $L_{i} < G_{i}$ such that the actions of $L_{i}$ on $W$ are recurrent and the actions of $S_{i} = L_{i} \cap S$ on $W$ are dissipative. Then for a fixed $i \in \{1,2\}$ the action $S_{i} \curvearrowright W$ admits a fundamental domain $U$. For any non-null Borel subset $U_{1} \subseteq gU$ for some $g \in S_{1}$, it is then clear that any $(L_{i} \curvearrowright W)|_{U_{1}}$ orbit contains infinitely many $(S_{i} \curvearrowright W)|_{U_{1}}$ orbits. Now writing any non-null Borel subset $E \subset W$ with finite measure as $E = \sqcup_{g \in S_{1}}gU \cap E$, we can see that each $(L_{i} \curvearrowright W)|_{E}$ orbit contains infinitely many $(S_{i} \curvearrowright W)|_{E}$ orbits, thus satisfying condition $(\ast)$. This gives the following corollary to Theorem \ref{Thm: on uniqueness of Cartan from Vaes}:

\begin{corollary}
\label{Corollary: on uniqueness of Cartans}
    Let $G = G_{1} *_{S} G_{2}$ be any amalgamated free product group and assume there exists $g_{1},...,g_{n} \in G$ such that $\cap_{i=1}^{n} g_{i}S g_{i}^{-1}$ is finite. Let $G \curvearrowright (W, \rho)$ be any nonsingular free ergodic action with Maharam extension $G \curvearrowright (\widetilde{W}, \widetilde{\rho})$. For $i \in \{1,2\}$, let $L_{i} < G_{i}$ be subgroups such that the actions $L_{i} \cap S \curvearrowright \widetilde{W}$ are dissipative and the actions $L_{i} \curvearrowright \widetilde{W}$ are recurrent. Then $L^{\infty}(W,\rho)$ is the unique Cartan subalgebra up to unitary conjugacy in $L^{\infty}(W,\rho)\rtimes G$.  
\end{corollary}

\begin{lemma}
\label{Lemma: III_1 equivalence relation satisfies condition ast}
    Let $\Gamma$ be as in Definition \ref{Const: Sigma, Lambda and Gamma} and $\Gamma \curvearrowright (Z_{1}, \eta_{1})$ as in Definition \ref{Const: III_1 action}. Then the Maharam extension $\Gamma \curvearrowright (Z, \eta)$ as in Definition \ref{Const: II infinity action} satisfies Condition $(\ast)$ of Theorem \ref{Thm: on uniqueness of Cartan from Vaes}.  
\end{lemma}
\begin{proof}
    Consider the subgroup $\Lambda_{1} \subset \Gamma_{1} = \SL(6,\mathbb{Z})$ given by a copy of $\SL(2,\mathbb{Z})$ in the first two rows and columns. Then we have that $\Sigma_{1} = \Lambda_{1} \cap \Sigma$ is given by:
    \begin{equation}
       \Sigma_{1} =  \left(\begin{array}{@{}c|c@{}}
  \begin{matrix}
      \epsilon_{1} & \mathbb{Z} \\ 0 &  \epsilon_{2}
  \end{matrix}
  & 0 \\
\hline
  0 &
  I
\end{array}\right) \text{ where } \epsilon_{1} = \epsilon_{2} = 1 \text{ or } -1
\end{equation}
It can be checked that the action of the top left $2 \times 2$ block on $\mathbb{R}^{2}$ is dissipative and hence the action of $\Sigma_{1} \curvearrowright Z$ is dissipative. Now we note that $\Lambda_{1} \curvearrowright \mathbb{R}^{2}$ is ergodic and look at the action on $Z$ as $\Lambda_{1} \curvearrowright (X \times \mathbb{R}^{4}) \times \mathbb{R}^{2}$. Choose a probability measure $\omega$ equivalent to the Lebesgue measure on $\mathbb{R}^{4}$ and note that the action $\Lambda_{1} \curvearrowright (X \times \mathbb{R}^{4}, \mu \times \omega)$ is still pmp. It follows from \cite[Theorem A.29]{AIM21} that the product of a recurrent action and a pmp action remains recurrent and hence the action $\Lambda_{1} \curvearrowright Z$ is recurrent. Now let $\Lambda_{2} = \SL(3,\mathbb{Z}) \subset \Gamma_{2}$ and note that $\Lambda_{2} \cap \Sigma = \{e\}$. Moreover $\Lambda_{2}$ is an infinite group with an essentially free pmp action on $Z$ and hence the action is recurrent. Now the result follows from Corollary \ref{Corollary: on uniqueness of Cartans}.  
\end{proof}

Now we have gathered all the ingredients to prove our main theorem. Let $\mathcal{R}_{1}$ be as in Definition \ref{Const: III_1 action} and let $M_{1} = L(\mathcal{R}_{1})$.  Recall that $\delta: \mathbb{R} \rightarrow \Out(M)$ denotes the canonical continuous homomorphism that does not depend on the choice of a faithful normal state.  

\begin{theorem}
\label{main theorem}
    Let $\mathcal{R}_{1}$ be as in Definition \ref{Const: III_1 action}. Then $M_{1} = L(\mathcal{R}_{1})$ is a full factor of type $\textrm{III}_{1}$. The modular homomorphism $\delta: \mathbb{R} \rightarrow \Out(M_{1})$ is an isomorphism and homeomorphism of Polish groups.
\end{theorem}

\begin{proof}
    By Lemma \ref{Lemma: type III1 and maharam ergodic}, we have that $M_{1}$ is a factor of type $\rm{III}_{1}$. We show first that every automorphism of $\cR_{1}$ is inner, i.e., $\Aut(\mathcal{R}_{1}) = [\cR_{1}]$. Indeed for any orbit equivalence $\Delta_{1} \in \Aut(\mathcal{R}_{1})$, we claim that $\Delta_{1}$ is inner. First note that $\Delta_{1}$ induces an infinite measure preserving orbit equivalence of the Maharam extension $\mathcal{R}$, let us call it $\Delta$. By Proposition \ref{Prop: every orbit equivalence is a conjugacy}, up to an inner automorphism $\phi$ of $\mathcal{R}$, we have that $\Delta(g \cdot z) = g \cdot \Delta(z)$ for all $g \in \Gamma$ and a.e. $z \in Z$. Now by Proposition \ref{Prop: every conjugacy is trivial}, we have that $\Delta = \phi \circ (\id \times \rho_{s})$ for some $s \in \mathbb{R} \backslash \{0\}$. Now, $\Delta$ restricted to $\mathcal{R}$ is equal to $\Delta_{1}$, and we have that $\Delta|_{\mathcal{R}_{1}} = \phi|_{\mathcal{R}_{1}} \circ (\id \times \rho_{s})|_{\mathcal{R}_{1}} = \Delta_{1}$. Note that  $\phi|_{\mathcal{R}_{1}} \in [\mathcal{R}_{1}]$, as for a.e. $z \in Z$ and $t \in \mathbb{R}$, we have that $((z,t),\phi(z,t)) \in \mathcal{R}$ and this implies $(z,\phi|_{\mathcal{R}_{1}}(z)) \in \mathcal{R}_{1}$. Let the restriction of $\rho_{s}: \mathbb{R}^{6} \rightarrow \mathbb{R}^{6}$ be denoted by $\theta: \mathbb{R}^{6}/\mathbb{R}^{*}_{+} \rightarrow \mathbb{R}^{6}/\mathbb{R}^{*}_{+}$. Note that if $s > 0$ then $\theta(y) = y$ for all $y \in \mathbb{R}^{6}/\mathbb{R}^{*}_{+}$ and if $s < 0$, then $\theta(y) = -y$ for all $y \in \mathbb{R}^{6}/\mathbb{R}^{*}_{+}$. In both cases, $(\id \times \theta) \in [\mathcal{R}_{1}]$ and this proves that $\Delta_{1} \in [\mathcal{R}_{1}]$. 
    
    Now, by Lemma \ref{Lemma: III_1 equivalence relation satisfies condition ast} and Theorem \ref{Thm: on uniqueness of Cartan from Vaes}, we have that $L^{\infty}(Z, \eta)$ is the unique Cartan subalgebra in $M_{1}$ up to unitary conjugacy. By Proposition \ref{Proposition: calculation of 1 cocycle group}, we have that $H^{1}(\mathcal{R}_{1},\mathbb{T})$ is equal to $\mathbb{R}$. Now fix a faithful normal state $\phi$ on $M_1$. Then $\sigma^{\phi}$ defines a continuous action of $\mathbb{R}$ on the Polish group $\cU(M_{1})$. Consider the semi-direct product $K = \cU(M_1) \rtimes_{\sigma^{\phi}} \mathbb{R}$ with its natural Polish group structure. Then the map 
    \begin{align*}
        \widetilde{\delta} : K \rightarrow \Aut(M_{1}) \text{ given by } (u,t) \mapsto \Ad(u) \circ \sigma^{\phi}_{t}
    \end{align*}
    is a continuous group homomorphism between Polish groups. Let $K_0$ be the kernel of $\widetilde{\delta}$, which is a closed normal subgroup of $K$. Since $M_1$ is of type III$_1$, we have that $K_0 = \mathbb{T} \cdot 1 \subset \cU(M_1)$. Thus $\widetilde{\delta}$ induces an injective and continuous group homomorphism $\widetilde{\delta}_{0} : K/K_{0} \rightarrow \Aut(M_1)$. We also have that $\widetilde{\delta}_{0}$ is surjective, by \cite[Theorem 3]{MR578730} and the fact that $\Aut(\cR_{1}) = [\cR_{1}]$. Hence $\widetilde{\delta}_{0}$ is a bijective continuous group homomorphism between Polish groups and by \cite[Lemma 3.4]{Connes74}, $\widetilde{\delta}_{0}$ is also a homeomorphism. Since $\cU(M_{1})/(\mathbb{T} \cdot 1)$ is closed in $K/K_{0}$, we have that $\Inn(M_1) = \widetilde{\delta}_{0}(\cU(M_1)/\mathbb{T} \cdot 1)$ is closed in $\Aut(M_1)$ and consequently $M_{1}$ is full. Denoting the canonical quotient by $\delta: \Aut(M_1) \rightarrow \Out(M_1)$, it follows that $\delta : \mathbb{R} \rightarrow \Out(M_{1})$ is an isomorphism and homeomorphism of Polish groups. 
\end{proof}

\printbibliography

@misc{deprez2012explicit,
      title={Explicit examples of equivalence relations and factors with prescribed fundamental group and outer automorphism group}, 
      author={Steven Deprez},
      year={2012},
      eprint={1010.3612},
      archivePrefix={arXiv},
      primaryClass={math.OA}
}

@article {MR578730,
    AUTHOR = {Feldman, Jacob and Moore, Calvin C.},
     TITLE = {Ergodic equivalence relations, cohomology, and von {N}eumann
              algebras. {II}},
   JOURNAL = {Trans. Amer. Math. Soc.},
  FJOURNAL = {Transactions of the American Mathematical Society},
    VOLUME = {234},
      YEAR = {1977},
    NUMBER = {2},
     PAGES = {325--359},
      ISSN = {0002-9947},
   MRCLASS = {22D40 (28A65 46L10)},
  MRNUMBER = {578730},
       DOI = {10.2307/1997925},
       URL = {https://doi.org/10.2307/1997925},
}

@incollection {PopaVaes11,
    AUTHOR = {Popa, Sorin and Vaes, Stefaan},
     TITLE = {Cocycle and orbit superrigidity for lattices in {${\rm
              SL}(n,\Bbb R)$} acting on homogeneous spaces},
 BOOKTITLE = {Geometry, rigidity, and group actions},
    SERIES = {Chicago Lectures in Math.},
     PAGES = {419--451},
 PUBLISHER = {Univ. Chicago Press, Chicago, IL},
      YEAR = {2011},
      ISBN = {978-0-226-23788-6; 0-226-23788-5},
   MRCLASS = {37A20 (22D10 22F10 46L36)},
  MRNUMBER = {2807839},
MRREVIEWER = {Alain\ Valette},
}

@article {MR578656,
    AUTHOR = {Feldman, Jacob and Moore, Calvin C.},
     TITLE = {Ergodic equivalence relations, cohomology, and von {N}eumann
              algebras. {I}},
   JOURNAL = {Trans. Amer. Math. Soc.},
  FJOURNAL = {Transactions of the American Mathematical Society},
    VOLUME = {234},
      YEAR = {1977},
    NUMBER = {2},
     PAGES = {289--324},
      ISSN = {0002-9947},
   MRCLASS = {22D40 (28A65 46L10)},
  MRNUMBER = {578656},
       DOI = {10.2307/1997924},
       URL = {https://doi.org/10.2307/1997924},
}

@phdthesis{lirias3869720,
    author = {Verjans, Bram},
    title = {W*- and orbit equivalence rigidity for group actions of type III},
    school = {KU Leuven},
    year = {2022}
}

@article {DrimbeVaes23,
    AUTHOR = {Drimbe, Daniel and Vaes, Stefaan},
     TITLE = {Superrigidity for dense subgroups of lie groups and their
              actions on homogeneous spaces},
   JOURNAL = {Math. Ann.},
  FJOURNAL = {Mathematische Annalen},
    VOLUME = {386},
      YEAR = {2023},
    NUMBER = {3-4},
     PAGES = {2015--2059},
      ISSN = {0025-5831,1432-1807},
   MRCLASS = {99-06},
  MRNUMBER = {4612412},
       DOI = {10.1007/s00208-022-02437-1},
       URL = {https://doi.org/10.1007/s00208-022-02437-1},
}

@incollection {Zimmer20,
    AUTHOR = {Zimmer, Robert J.},
     TITLE = {Strong rigidity for ergodic actions of semisimple {L}ie
              groups},
 BOOKTITLE = {Group actions in ergodic theory, geometry, and
              topology---selected papers},
     PAGES = {183--201},
      NOTE = {Reprint of [0595205]},
 PUBLISHER = {Univ. Chicago Press, Chicago, IL},
      YEAR = {2020},
      ISBN = {978-0-226-56813-3; 978-0-226-56827-0},
   MRCLASS = {},
  MRNUMBER = {4521477},
}

@article {MR2342637,
    AUTHOR = {Popa, Sorin},
     TITLE = {Cocycle and orbit equivalence superrigidity for malleable
              actions of {$w$}-rigid groups},
   JOURNAL = {Invent. Math.},
  FJOURNAL = {Inventiones Mathematicae},
    VOLUME = {170},
      YEAR = {2007},
    NUMBER = {2},
     PAGES = {243--295},
      ISSN = {0020-9910,1432-1297},
   MRCLASS = {37A20 (28D15 37A15 37A35 46L55)},
  MRNUMBER = {2342637},
MRREVIEWER = {Alain\ Valette},
       DOI = {10.1007/s00222-007-0063-0},
       URL = {https://doi.org/10.1007/s00222-007-0063-0},
}

@article {MR0049194,
    AUTHOR = {Hua, L. K. and Reiner, I.},
     TITLE = {Automorphisms of the projective unimodular group},
   JOURNAL = {Trans. Amer. Math. Soc.},
  FJOURNAL = {Transactions of the American Mathematical Society},
    VOLUME = {72},
      YEAR = {1952},
     PAGES = {467--473},
      ISSN = {0002-9947,1088-6850},
   MRCLASS = {20.0X},
  MRNUMBER = {49194},
MRREVIEWER = {E.\ Grosswald},
       DOI = {10.2307/1990713},
       URL = {https://doi.org/10.2307/1990713},
}

@article {PopaVaes08,
    AUTHOR = {Popa, Sorin and Vaes, Stefaan},
     TITLE = {Strong rigidity of generalized {B}ernoulli actions and
              computations of their symmetry groups},
   JOURNAL = {Adv. Math.},
  FJOURNAL = {Advances in Mathematics},
    VOLUME = {217},
      YEAR = {2008},
    NUMBER = {2},
     PAGES = {833--872},
      ISSN = {0001-8708,1090-2082},
   MRCLASS = {37A15 (22F10 46L10 46L55)},
  MRNUMBER = {2370283},
MRREVIEWER = {Alain\ Valette},
       DOI = {10.1016/j.aim.2007.09.006},
       URL = {https://doi.org/10.1016/j.aim.2007.09.006},
}

@article {MR3048005,
    AUTHOR = {Keersmaekers, J. and Speelman, A.},
     TITLE = {{${\rm II}_1$} factors and equivalence relations with distinct
              fundamental groups},
   JOURNAL = {Internat. J. Math.},
  FJOURNAL = {International Journal of Mathematics},
    VOLUME = {24},
      YEAR = {2013},
    NUMBER = {3},
     PAGES = {1350016, 24},
      ISSN = {0129-167X,1793-6519},
   MRCLASS = {46L10},
  MRNUMBER = {3048005},
MRREVIEWER = {S\'{e}bastien\ Falgui\`eres},
       DOI = {10.1142/S0129167X1350016X},
       URL = {https://doi.org/10.1142/S0129167X1350016X},
}

@article {MR0675419,
    AUTHOR = {Schmidt, Klaus and Walters, Peter},
     TITLE = {Mildly mixing actions of locally compact groups},
   JOURNAL = {Proc. London Math. Soc. (3)},
  FJOURNAL = {Proceedings of the London Mathematical Society. Third Series},
    VOLUME = {45},
      YEAR = {1982},
    NUMBER = {3},
     PAGES = {506--518},
      ISSN = {0024-6115,1460-244X},
   MRCLASS = {28D15 (22D40)},
  MRNUMBER = {675419},
MRREVIEWER = {S.\ G.\ Dani},
       DOI = {10.1112/plms/s3-45.3.506},
       URL = {https://doi.org/10.1112/plms/s3-45.3.506},
}

@article {AIM21,
    AUTHOR = {Arano, Yuki and Isono, Yusuke and Marrakchi, Amine},
     TITLE = {Ergodic theory of affine isometric actions on {H}ilbert
              spaces},
   JOURNAL = {Geom. Funct. Anal.},
  FJOURNAL = {Geometric and Functional Analysis},
    VOLUME = {31},
      YEAR = {2021},
    NUMBER = {5},
     PAGES = {1013--1094},
      ISSN = {1016-443X,1420-8970},
   MRCLASS = {37A40 (20E08 20F65 28C20 37A50)},
  MRNUMBER = {4356699},
MRREVIEWER = {Eusebio\ Gardella},
       DOI = {10.1007/s00039-021-00584-2},
       URL = {https://doi.org/10.1007/s00039-021-00584-2},
}

@article {Vaes14,
    AUTHOR = {Vaes, Stefaan},
     TITLE = {Normalizers inside amalgamated free product von {N}eumann
              algebras},
   JOURNAL = {Publ. Res. Inst. Math. Sci.},
  FJOURNAL = {Publications of the Research Institute for Mathematical
              Sciences},
    VOLUME = {50},
      YEAR = {2014},
    NUMBER = {4},
     PAGES = {695--721},
      ISSN = {0034-5318,1663-4926},
   MRCLASS = {46L36 (37A40 46L54)},
  MRNUMBER = {3273307},
MRREVIEWER = {R\'{e}mi\ Boutonnet},
       DOI = {10.4171/PRIMS/147},
       URL = {https://doi.org/10.4171/PRIMS/147},
}

@article {DeprezVaes11,
    AUTHOR = {Deprez, Steven and Vaes, Stefaan},
     TITLE = {A classification of all finite index subfactors for a class of
              group-measure space {${\rm II}_1$} factors},
   JOURNAL = {J. Noncommut. Geom.},
  FJOURNAL = {Journal of Noncommutative Geometry},
    VOLUME = {5},
      YEAR = {2011},
    NUMBER = {4},
     PAGES = {523--545},
      ISSN = {1661-6952,1661-6960},
   MRCLASS = {46L37 (28D15 46L36)},
  MRNUMBER = {2838524},
MRREVIEWER = {Junsheng\ Fang},
       DOI = {10.4171/JNCG/85},
       URL = {https://doi.org/10.4171/JNCG/85},
}

@article {IKT09,
    AUTHOR = {Ioana, Adrian and Kechris, Alexander S. and Tsankov, Todor},
     TITLE = {Subequivalence relations and positive-definite functions},
   JOURNAL = {Groups Geom. Dyn.},
  FJOURNAL = {Groups, Geometry, and Dynamics},
    VOLUME = {3},
      YEAR = {2009},
    NUMBER = {4},
     PAGES = {579--625},
      ISSN = {1661-7207,1661-7215},
   MRCLASS = {37A20 (03E15 20B99 28D05 37A25)},
  MRNUMBER = {2529949},
MRREVIEWER = {Konstantin\ Medynets},
       DOI = {10.4171/GGD/71},
       URL = {https://doi.org/10.4171/GGD/71},
}

@article {MR0103421,
    AUTHOR = {Blattner, Robert J.},
     TITLE = {Automorphic group representations},
   JOURNAL = {Pacific J. Math.},
  FJOURNAL = {Pacific Journal of Mathematics},
    VOLUME = {8},
      YEAR = {1958},
     PAGES = {665--677},
      ISSN = {0030-8730,1945-5844},
   MRCLASS = {46.00},
  MRNUMBER = {103421},
MRREVIEWER = {F.\ I.\ Mautner},
       URL = {http://projecteuclid.org/euclid.pjm/1103039692},
}

@article {MR0587372,
    AUTHOR = {Connes, A.},
     TITLE = {A factor of type {${\rm II}\sb{1}$} with countable fundamental
              group},
   JOURNAL = {J. Operator Theory},
  FJOURNAL = {Journal of Operator Theory},
    VOLUME = {4},
      YEAR = {1980},
    NUMBER = {1},
     PAGES = {151--153},
      ISSN = {0379-4024},
   MRCLASS = {46L10 (22D35)},
  MRNUMBER = {587372},
MRREVIEWER = {Richard\ I.\ Loebl},
}

@article {MR2386109,
    AUTHOR = {Ioana, Adrian and Peterson, Jesse and Popa, Sorin},
     TITLE = {Amalgamated free products of weakly rigid factors and
              calculation of their symmetry groups},
   JOURNAL = {Acta Math.},
  FJOURNAL = {Acta Mathematica},
    VOLUME = {200},
      YEAR = {2008},
    NUMBER = {1},
     PAGES = {85--153},
      ISSN = {0001-5962,1871-2509},
   MRCLASS = {46L10 (20E06 22D25 37A15 37A20 46L09)},
  MRNUMBER = {2386109},
MRREVIEWER = {Alain\ Valette},
       DOI = {10.1007/s11511-008-0024-5},
       URL = {https://doi.org/10.1007/s11511-008-0024-5},
}

@article {MR2409162,
    AUTHOR = {Falgui\`eres, S\'{e}bastien and Vaes, Stefaan},
     TITLE = {Every compact group arises as the outer automorphism group of
              a {${\rm II}_1$} factor},
   JOURNAL = {J. Funct. Anal.},
  FJOURNAL = {Journal of Functional Analysis},
    VOLUME = {254},
      YEAR = {2008},
    NUMBER = {9},
     PAGES = {2317--2328},
      ISSN = {0022-1236,1096-0783},
   MRCLASS = {46L10 (46L35)},
  MRNUMBER = {2409162},
MRREVIEWER = {Stuart\ A.\ White},
       DOI = {10.1016/j.jfa.2008.02.002},
       URL = {https://doi.org/10.1016/j.jfa.2008.02.002},
}

@article {MR2601038,
    AUTHOR = {Popa, Sorin and Vaes, Stefaan},
     TITLE = {Actions of {$\Bbb F_\infty$} whose {${\rm II}_1$} factors and
              orbit equivalence relations have prescribed fundamental group},
   JOURNAL = {J. Amer. Math. Soc.},
  FJOURNAL = {Journal of the American Mathematical Society},
    VOLUME = {23},
      YEAR = {2010},
    NUMBER = {2},
     PAGES = {383--403},
      ISSN = {0894-0347,1088-6834},
   MRCLASS = {46L10 (28D15 37A20)},
  MRNUMBER = {2601038},
MRREVIEWER = {Narutaka\ Ozawa},
       DOI = {10.1090/S0894-0347-09-00644-4},
       URL = {https://doi.org/10.1090/S0894-0347-09-00644-4},
}

@article {MR2504433,
    AUTHOR = {Vaes, Stefaan},
     TITLE = {Explicit computations of all finite index bimodules for a
              family of {${\rm II}_1$} factors},
   JOURNAL = {Ann. Sci. \'{E}c. Norm. Sup\'{e}r. (4)},
  FJOURNAL = {Annales Scientifiques de l'\'{E}cole Normale Sup\'{e}rieure.
              Quatri\`eme S\'{e}rie},
    VOLUME = {41},
      YEAR = {2008},
    NUMBER = {5},
     PAGES = {743--788},
      ISSN = {0012-9593,1873-2151},
   MRCLASS = {46L10 (46L37 46L40)},
  MRNUMBER = {2504433},
MRREVIEWER = {Mart\'{\i}n\ Argerami},
       DOI = {10.24033/asens.2081},
       URL = {https://doi.org/10.24033/asens.2081},
}

@article {MR3335839,
    AUTHOR = {Ioana, Adrian},
     TITLE = {Cartan subalgebras of amalgamated free product {${\rm II}_1$}
              factors},
      NOTE = {With an appendix by Ioana and Stefaan Vaes},
   JOURNAL = {Ann. Sci. \'{E}c. Norm. Sup\'{e}r. (4)},
  FJOURNAL = {Annales Scientifiques de l'\'{E}cole Normale Sup\'{e}rieure.
              Quatri\`eme S\'{e}rie},
    VOLUME = {48},
      YEAR = {2015},
    NUMBER = {1},
     PAGES = {71--130},
      ISSN = {0012-9593,1873-2151},
   MRCLASS = {46L36 (28D15 37A20 46L10 60B15)},
  MRNUMBER = {3335839},
MRREVIEWER = {Sven\ Raum},
       DOI = {10.24033/asens.2239},
       URL = {https://doi.org/10.24033/asens.2239},
}

@article {MR3102166,
    AUTHOR = {Houdayer, Cyril and Vaes, Stefaan},
     TITLE = {Type {III} factors with unique {C}artan decomposition},
   JOURNAL = {J. Math. Pures Appl. (9)},
  FJOURNAL = {Journal de Math\'{e}matiques Pures et Appliqu\'{e}es.
              Neuvi\`eme S\'{e}rie},
    VOLUME = {100},
      YEAR = {2013},
    NUMBER = {4},
     PAGES = {564--590},
      ISSN = {0021-7824,1776-3371},
   MRCLASS = {46L10 (37A40 46L36)},
  MRNUMBER = {3102166},
MRREVIEWER = {Bachir\ Bekka},
       DOI = {10.1016/j.matpur.2013.01.013},
       URL = {https://doi.org/10.1016/j.matpur.2013.01.013},
}

@article {MR3164361,
    AUTHOR = {Boutonnet, R\'{e}mi and Houdayer, Cyril and Raum, Sven},
     TITLE = {Amalgamated free product type {III} factors with at most one
              {C}artan subalgebra},
   JOURNAL = {Compos. Math.},
  FJOURNAL = {Compositio Mathematica},
    VOLUME = {150},
      YEAR = {2014},
    NUMBER = {1},
     PAGES = {143--174},
      ISSN = {0010-437X,1570-5846},
   MRCLASS = {46L10 (37A40 46L54)},
  MRNUMBER = {3164361},
MRREVIEWER = {Wenming\ Wu},
       DOI = {10.1112/S0010437X13007537},
       URL = {https://doi.org/10.1112/S0010437X13007537},
}

@article {MR2425177,
    AUTHOR = {Popa, Sorin},
     TITLE = {On the superrigidity of malleable actions with spectral gap},
   JOURNAL = {J. Amer. Math. Soc.},
  FJOURNAL = {Journal of the American Mathematical Society},
    VOLUME = {21},
      YEAR = {2008},
    NUMBER = {4},
     PAGES = {981--1000},
      ISSN = {0894-0347,1088-6834},
   MRCLASS = {46L35 (22D25 22D40 28D15 37A20)},
  MRNUMBER = {2425177},
MRREVIEWER = {Claire\ Anantharaman-Delaroche},
       DOI = {10.1090/S0894-0347-07-00578-4},
       URL = {https://doi.org/10.1090/S0894-0347-07-00578-4},
}

@article {MR2783933,
    AUTHOR = {Ioana, Adrian},
     TITLE = {Cocycle superrigidity for profinite actions of property ({T})
              groups},
   JOURNAL = {Duke Math. J.},
  FJOURNAL = {Duke Mathematical Journal},
    VOLUME = {157},
      YEAR = {2011},
    NUMBER = {2},
     PAGES = {337--367},
      ISSN = {0012-7094,1547-7398},
   MRCLASS = {37A20 (22F10 28D15 46L36)},
  MRNUMBER = {2783933},
MRREVIEWER = {Alain\ Valette},
       DOI = {10.1215/00127094-2011-008},
       URL = {https://doi.org/10.1215/00127094-2011-008},
}

@article {MR4664830,
    AUTHOR = {Vaes, Stefaan and Verjans, Bram},
     TITLE = {Orbit equivalence superrigidity for type {III} actions},
   JOURNAL = {Ergodic Theory Dynam. Systems},
  FJOURNAL = {Ergodic Theory and Dynamical Systems},
    VOLUME = {43},
      YEAR = {2023},
    NUMBER = {12},
     PAGES = {4193--4225},
      ISSN = {0143-3857,1469-4417},
   MRCLASS = {99-06},
  MRNUMBER = {4664830},
       DOI = {10.1017/etds.2022.112},
       URL = {https://doi.org/10.1017/etds.2022.112},
}

@article{85852f12-3669-343b-b6ad-176dfe8edb82,
 ISSN = {08940347, 10886834},
 URL = {http://www.jstor.org/stable/20161375},
 author = {Alex Furman},
 journal = {Journal of the American Mathematical Society},
 number = {2},
 pages = {479--512},
 publisher = {American Mathematical Society},
 title = {Measurable Rigidity of Actions on Infinite Measure Homogeneous Spaces, II},
 urldate = {2023-12-06},
 volume = {21},
 year = {2008}
}

@misc{PetersonLectureNotes11,
  author        = {Jesse Peterson},
  title         = {Lecture notes on ergodic theory},
  month         = {March 17},
  year          = {2011},
  university    = {Vanderbilt University},  
  URL={https://math.vanderbilt.edu/peters10/teaching/Spring2011/ErgodicTheoryNotes.pdf}
}

@book {Aaronson97,
    AUTHOR = {Aaronson, Jon},
     TITLE = {An introduction to infinite ergodic theory},
    SERIES = {Mathematical Surveys and Monographs},
    VOLUME = {50},
 PUBLISHER = {American Mathematical Society, Providence, RI},
      YEAR = {1997},
     PAGES = {xii+284},
      ISBN = {0-8218-0494-4},
   MRCLASS = {28Dxx (58F11 58F17)},
  MRNUMBER = {1450400},
MRREVIEWER = {Cesar\ E.\ Silva},
       DOI = {10.1090/surv/050},
       URL = {https://doi.org/10.1090/surv/050},
}

@article {Popa07,
    AUTHOR = {Popa, Sorin},
     TITLE = {Cocycle and orbit equivalence superrigidity for malleable
              actions of {$w$}-rigid groups},
   JOURNAL = {Invent. Math.},
  FJOURNAL = {Inventiones Mathematicae},
    VOLUME = {170},
      YEAR = {2007},
    NUMBER = {2},
     PAGES = {243--295},
      ISSN = {0020-9910,1432-1297},
   MRCLASS = {37A20 (28D15 37A15 37A35 46L55)},
  MRNUMBER = {2342637},
MRREVIEWER = {Alain\ Valette},
       DOI = {10.1007/s00222-007-0063-0},
       URL = {https://doi.org/10.1007/s00222-007-0063-0},
}

@article {Popa06,
    AUTHOR = {Popa, Sorin},
     TITLE = {Some rigidity results for non-commutative {B}ernoulli shifts},
   JOURNAL = {J. Funct. Anal.},
  FJOURNAL = {Journal of Functional Analysis},
    VOLUME = {230},
      YEAR = {2006},
    NUMBER = {2},
     PAGES = {273--328},
      ISSN = {0022-1236,1096-0783},
   MRCLASS = {46L10 (37A20 46L40 46L55)},
  MRNUMBER = {2186215},
MRREVIEWER = {Stefaan\ Vaes},
       DOI = {10.1016/j.jfa.2005.06.017},
       URL = {https://doi.org/10.1016/j.jfa.2005.06.017},
}

@article {Popa06StrongRigidity1,
    AUTHOR = {Popa, Sorin},
     TITLE = {Strong rigidity of {$\rm II_1$} factors arising from malleable
              actions of {$w$}-rigid groups. {I}},
   JOURNAL = {Invent. Math.},
  FJOURNAL = {Inventiones Mathematicae},
    VOLUME = {165},
      YEAR = {2006},
    NUMBER = {2},
     PAGES = {369--408},
      ISSN = {0020-9910,1432-1297},
   MRCLASS = {46L10 (22D25 37A20 46L55)},
  MRNUMBER = {2231961},
MRREVIEWER = {Alain\ Valette},
       DOI = {10.1007/s00222-006-0501-4},
       URL = {https://doi.org/10.1007/s00222-006-0501-4},
}

@article {Popa06StrongRigidity2,
    AUTHOR = {Popa, Sorin},
     TITLE = {Strong rigidity of {$\rm II_1$} factors arising from malleable
              actions of {$w$}-rigid groups. {II}},
   JOURNAL = {Invent. Math.},
  FJOURNAL = {Inventiones Mathematicae},
    VOLUME = {165},
      YEAR = {2006},
    NUMBER = {2},
     PAGES = {409--451},
      ISSN = {0020-9910,1432-1297},
   MRCLASS = {46L55 (22D25 37A20 37A55 46L10)},
  MRNUMBER = {2231962},
MRREVIEWER = {Alain\ Valette},
       DOI = {10.1007/s00222-006-0502-3},
       URL = {https://doi.org/10.1007/s00222-006-0502-3},
}

@article {Connes-Takesaki_1977,
    AUTHOR = {Connes, Alain and Takesaki, Masamichi},
     TITLE = {The flow of weights on factors of type {${\rm III}$}},
   JOURNAL = {Tohoku Math. J. (2)},
  FJOURNAL = {The Tohoku Mathematical Journal. Second Series},
    VOLUME = {29},
      YEAR = {1977},
    NUMBER = {4},
     PAGES = {473--575},
      ISSN = {0040-8735,2186-585X},
   MRCLASS = {46L10 (46L35 46L55)},
  MRNUMBER = {480760},
MRREVIEWER = {M. Walter},
       DOI = {10.2748/tmj/1178240493},
       URL = {https://doi.org/10.2748/tmj/1178240493},
}

@book {SunderAnInvitation,
    AUTHOR = {Sunder, V. S.},
     TITLE = {An invitation to von {N}eumann algebras},
    SERIES = {Universitext},
 PUBLISHER = {Springer-Verlag, New York},
      YEAR = {1987},
     PAGES = {xiv+171},
      ISBN = {0-387-96356-1},
   MRCLASS = {46L10 (46-01)},
  MRNUMBER = {866671},
MRREVIEWER = {D.\ Petz},
       DOI = {10.1007/978-1-4613-8669-8},
       URL = {https://doi.org/10.1007/978-1-4613-8669-8},
}

@article {Haagerup75,
    AUTHOR = {Haagerup, Uffe},
     TITLE = {The standard form of von {N}eumann algebras},
   JOURNAL = {Math. Scand.},
  FJOURNAL = {Mathematica Scandinavica},
    VOLUME = {37},
      YEAR = {1975},
    NUMBER = {2},
     PAGES = {271--283},
      ISSN = {0025-5521,1903-1807},
   MRCLASS = {46L10},
  MRNUMBER = {407615},
MRREVIEWER = {H.\ Araki},
       DOI = {10.7146/math.scand.a-11606},
       URL = {https://doi.org/10.7146/math.scand.a-11606},
}

@article {Connes74,
    AUTHOR = {Connes, A.},
     TITLE = {Almost periodic states and factors of type {${\rm
              III}\sb{1}$}},
   JOURNAL = {J. Functional Analysis},
  FJOURNAL = {Journal of Functional Analysis},
    VOLUME = {16},
      YEAR = {1974},
     PAGES = {415--445},
      ISSN = {0022-1236},
   MRCLASS = {46L10},
  MRNUMBER = {358374},
MRREVIEWER = {E.\ St\o rmer},
       DOI = {10.1016/0022-1236(74)90059-7},
       URL = {https://doi.org/10.1016/0022-1236(74)90059-7},
}

@article {Moore76,
    AUTHOR = {Moore, Calvin C.},
     TITLE = {Group extensions and cohomology for locally compact groups.
              {III}},
   JOURNAL = {Trans. Amer. Math. Soc.},
  FJOURNAL = {Transactions of the American Mathematical Society},
    VOLUME = {221},
      YEAR = {1976},
    NUMBER = {1},
     PAGES = {1--33},
      ISSN = {0002-9947,1088-6850},
   MRCLASS = {22D05 (22D10 22D30)},
  MRNUMBER = {414775},
MRREVIEWER = {J.\ M. G. Fell},
       DOI = {10.2307/1997540},
       URL = {https://doi.org/10.2307/1997540},
}

@article {ConnesKrieger77,
    AUTHOR = {Connes, Alain and Krieger, Wolfgang},
     TITLE = {Measure space automorphisms, the normalizers of their full
              groups, and approximate finiteness},
   JOURNAL = {J. Functional Analysis},
    VOLUME = {24},
      YEAR = {1977},
    NUMBER = {4},
     PAGES = {336--352},
   MRCLASS = {28A65},
  MRNUMBER = {0444900},
MRREVIEWER = {Douglas Lind},
       DOI = {10.1016/0022-1236(77)90062-3},
       URL = {https://doi.org/10.1016/0022-1236(77)90062-3},
}

@misc{CalderoniLectureNotes09,
  AUTHOR        = {Calderoni, F.},
  TITLE         = {Lecture notes on countable Borel equivalence relations},
  SCHOOL = {University of Illinois, Chicago},  
  MONTH         = {Spring},
  YEAR          = {2020},
  URL={https://math.berkeley.edu/~vfr/VonNeumann2009.pdf}
}

@incollection {Weinberger97,
    AUTHOR = {Weinberger, Shmuel},
     TITLE = {{${\rm SL}(n,\bold Z)$} cannot act on small tori},
 BOOKTITLE = {Geometric topology ({A}thens, {GA}, 1993)},
    SERIES = {AMS/IP Stud. Adv. Math.},
    VOLUME = {2.1},
     PAGES = {406--408},
 PUBLISHER = {Amer. Math. Soc., Providence, RI},
      YEAR = {1997},
      ISBN = {0-8218-0654-8},
   MRCLASS = {57S25},
  MRNUMBER = {1470739},
       DOI = {10.1090/amsip/002.1/22},
       URL = {https://doi.org/10.1090/amsip/002.1/22},
}

@misc{DeMedtsLectureNotes,
  author        = {Tom De Medts},
  title         = {Lecture notes on linear algebraic groups},
  university    = {Department of Mathematics, Ghent University},
  year          = {2022-2023},
  URL={https://algebra.ugent.be/~tdemedts/files/LinearAlgebraicGroups-TomDeMedts.pdf}
}
\end{document}